\documentclass{amsart}
\usepackage{amsmath,amscd,amssymb}
\usepackage{amsfonts}
\usepackage[all]{xy}
\usepackage{enumerate}
\usepackage{color}
\numberwithin{equation}{subsection}
\theoremstyle{plain}
\newtheorem{thm}[subsection]{Theorem}
\newtheorem{prop}[subsection]{Proposition}
\newtheorem{lemma}[subsection]{Lemma}
\newtheorem{cor}[subsection]{Corollary}

\theoremstyle{definition}
\newtheorem{defn}[subsection]{Definition}
\newtheorem{notn}[subsection]{Notation}
\newtheorem{cont}[subsection]{Contents}
\newtheorem{ackn}[subsection]{Acknowledgement}
\theoremstyle{remark}
\newtheorem{rem}[subsection]{Remark}
\newtheorem{rems}[subsection]{Remarks}

\begin{document}
\title{Simplicial homotopy theory of algebraic varieties over real closed fields, Part 1}
\author{Ambrus P\'al}
\date{July 4, 2022.}
\address{Department of Mathematics, 180 Queen's Gate, Imperial College, London, SW7 2AZ, United Kingdom}
\email{a.pal@imperial.ac.uk}
\begin{abstract} We study the homotopy type of the simplicial set of continuous semi-algebraic simplexes of an algebraic variety defined over a real closed field, which we will call the real homotopy type. We prove an analogue of the theorem of Artin-Mazur comparing the real homotopy type with the \'etale homotopy type. This paper is part one of a sequence of papers on this topic.
\end{abstract}
\maketitle
\markboth{Ambrus P\'al}{Simplicial homotopy over real closed fields}

\tableofcontents

\section{Introduction}

Let $K$ be a real closed field and let $C$ denote its algebraic closure. The aim of this paper is to study the {\it real (or definable) homotopy type} of a scheme $X$ over such a $K$. It is defined as follows. Let $\Delta^n_K$ denote the $n$-dimensional simplex over $K$:
$$\Delta^n_K=\{(x_0,x_1,\ldots,x_n)\in K^n\mid \sum_{i=0}^nx_i=1\textrm{ and }0\leq x_i\textrm{ }(\forall i=0,1,\ldots,n)\}$$
where $\leq$ denotes the canonical ordering of $K$. Let $Y\subseteq K^n$ be a semi-algebraic set. By the {\it simplicial set of semi-algebraic singular simplexes} of $Y$ we mean the simplicial set $S_*(Y)$ such that for every $n\in\mathbb N$ the set $S_n(Y)$ consists of all continuous semi-algebraic maps $f:\Delta^n_K\rightarrow Y$ while the face and degeneracy maps of $S_*(Y)$ are defined the same way as the usual ones. Let $Re(Y)$ denote the homotopy type of $S_*(Y)$ in  the homotopy category Ho$(SSets)$ of simplicial sets. 

The key result relating the real homotopy type with the relative etale homotopy type is the analogue of a basic comparison theorem of Artin and Mazur. Let $Sch$ denote the category of locally Noetherian schemes, and let
$$Et(\cdot):Sch\longrightarrow\textrm{Pro-}\textrm{Ho}(Sets)$$
denote the Artin--Mazur \'etale homotopy type functor. For every object $Y$ in  Ho$(SSets)$ let $X^{\wedge}$ denote its profinite completion in 
$\textrm{\rm Pro-}\textrm{\rm Ho}(SSets)$.
\begin{thm}\label{artin-mazur} For every quasi-projective geometrically unibranch variety $X$ over $C$ there is a natural weak equivalence:
$$m_X:Et(X)\longrightarrow Re(X(C))^{\wedge}
$$
in $\textrm{\rm Pro-}\textrm{\rm Ho}(SSets)$.
\end{thm}
Just like in the classical case, the proof of the result above relies on a semi-algebraic version of the Riemann existence theorem (see Theorem \ref{riemann_existence}) and a comparison isomorphism between the \'etale and the semi-algebraic cohomology groups of a finite locally constant \'etale sheaf (see Theorem \ref{cohomological}). Both of these result has been published in Roland Huber's doctoral dissertation \cite{{Hu}}. In a true {\it tour de force} Huber develops enough of the theory of the semi-algebraic analogue of complex analytic varieties and GAGA type theorems to prove these results in the large generality we want to use them. Unfortunately this thesis was never fully published; in addition to some announcements of results (see \cite{HK} and \cite{Sch}), only parts appeared in \cite{Hu2}.Therefore in this paper we will apply a somewhat curious approach: instead of trying to cover Huber's work, we will give rather different proofs, but in the smooth case only. This way the sceptical reader could at least accept the validity of our main result in the smooth case.

The proof of the central comparison result (Theorem \ref{artin-mazur}) largely follows the original arguments. As we mentioned above the key ingredients are the analogue of the Riemann existence theorem (Theorem \ref{riemann_existence}), and an SGA-type comparison result for cohomology (Theorem \ref{cohomological}). Our arguments in the special case of smooth varieties mostly differ only in the aspect of avoiding the use of GAGA-type theorems, but instead computing the semi-algebraic fundamental group and cohomology of elementary fibrations. The key tool to do so is Theorem \ref{fibration-theorem2}, whose content is that maps of algebraic varieties over $C$ which are elementary fibrations in the geometric sense are Kan fibrations in real homotopy theory. Once these basic comparison theorems are available the proof of Theorem \ref{artin-mazur} is essentially the same as the original one.

It is worth noting that it is possible to define the topological type and therefore the homotopy type of definable sets in a more general situation, namely for $o$-minimal theories. The cohomology theory of definable sets for $o$-minimal theories is a highly developed subject, most notably in the work of Edmondo, Woerheide and their collaborators (see for example \cite{Wo}, where the simplicial set of semi-algebraic singular simplexes the appears for the first time, but see also \cite{EW} and \cite{EP} where the Grothendieck's formalism of $6$ operations is developed). However we are mostly interested in comparison results with \'etale homotopy, so for our purposes it is natural to restrict to the theory of real closed fields. In this case we could rely on the rather extensive results of Delfs and Knebusch (see \cite{DeKn1} and \cite{DeKn2}).
\begin{cont} In the next section we introduce the topological type and the definable homotopy type of locally semi-algebraic spaces, and prove some basic comparison theorems: we show the invariance of the definable homotopy type under extension of the base field, and we prove that the definable homotopy type is isomorphic to the usual one when $K$ is the real number field.

In sections three to six we prove, modulo some technical results to be published in the second part, our main technical result, the fibration theorem: it claims that a geometric fibration of smooth varieties over $C$, i.e.~a smooth map with good compactification, induces a Kan fibration of the simplicial sets of semi-algebraic singular simplexes. As an application we prove that the definable homotopy type of smooth curves and abelian varieties are Eilenberg--MacLane spaces in the seventh section. In the eighth section we prove the analogue of the Riemann existence theorem, first for curves, then for elementary fibrations with the help of the fibration theorem, and finally in all cases using the standard arguments in SGA.

In the ninth section we prove the cohomological comparison theorem; again the essential case is elementary fibrations, when the result follows from the fibration theorem, while the rest of the argument is the same as in the classical result. We show our generalisation of the Artin--Mazur comparison theorem in the tenth version, basically along the lines of the original argument. \end{cont}
\begin{ackn} The author was partially supported by the EPSRC grant P36794. I also wish to thank Endre Szab\'o for useful discussions concerning the contents of this paper.
\end{ackn}

\section{The topological type of locally semi-algebraic spaces}

We start this section by recalling some basic definitions of \cite{DeKn1} and \cite{DeKn2}. For the basic notions of simplicial homotopy theory used here see \cite{GJ}. 
\begin{defn}\label{2.1} By a semi-algebraic set (over $K$) we mean a subset $V\subset K^m$ for some $m\in\mathbb N$ which is definable. The ordering of $K$ induces a topology on the set $K^m$, and hence on every semi-algebraic subset $V$ of $K^m$. We call this topology the strong topology. A map $f:V\to W$ between two semi-algebraic sets $V\subset K^m,W\subset K^n$ is called semi-algebraic, if it is continuous with respect to the strong topology and if the graph $\Gamma_f\subset K^n\times K^m=K^{n+m}$ of $f$ is a semi-algebraic subset. Since composition of semi-algebraic maps is semi-algebraic we get that the class of
semi-algebraic sets form a category $\mathcal{SA}_K$ where morphisms are semi-algebraic maps.
\end{defn}
\begin{defn}\label{2.2} For every semi-algebraic set $V$ we define a Grothendieck topology $\mathfrak{S}(V)$ consisting of inclusion maps of subsets of $V$ as follows. Open subsets of $\mathfrak{S}(V)$ are semi-algebraic subsets of $V$ which are open with respect to the strong topology, and coverings are systems which could be refined to finite systems of inclusions $\{V_i\to V\}_{i\in I}$ such that $\bigcup_{i\in I}V_i=V$. Let $\mathcal O_V$ denote the sheaf of $K$-valued functions on $\mathfrak{S}(V)$ which assigns to every open $U\subset V$ the ring of semi-algebraic functions $f:U\to K$. Equipped with this extra structure the triple $(V,\mathfrak{S}(V),\mathcal O_V)$ is a (locally) ringed space which we will call the semi-algebraic space attached to $V$. We may identify $\mathcal{SA}_K$ with a full subcategory of ringed spaces via the functor $V\mapsto (V,\mathfrak{S}(V),\mathcal O_V)$. 
\end{defn}
\begin{defn}\label{2.3} We say that a ringed space $(M,\mathfrak{S}(M),\mathcal O_M)$ is a locally semi-algebraic space (over $K$), or more concisely an $lsa$-space, if $\mathfrak{S}(M)$ consists of inclusion maps of subsets of $M$, and $M$ has an admissible covering $\{M_{\alpha}\}_{\alpha\in I}$ such that for every $\alpha\in I$ the ringed space
$(M_{\alpha},\mathcal O_M|_{M_{\alpha}})$ is a semi-algebraic space (attached to a semi-algebraic set). Let $\mathcal{LSA}_K$ denote the full subcategory of ringed spaces whose objects are locally semi-algebraic spaces. Every morphism $(M,\mathfrak{S}(M),\mathcal O_M)\to
(N,\mathfrak{S}(N),\mathcal O_N)$ is uniquely determined by the map $M\to N$ on the underlying sets by Theorem 1.2 of \cite{DeKn2} on page 7, it will not cause much confusion when we will usually let $M$ denote the tripe $(M,\mathfrak{S}(M),\mathcal O_M)$ of the type above. For simplicity we will call morphisms in $\mathcal{LSA}_K$ semi-algebraic maps.
\end{defn}
Now we are ready to introduce the basic construction of this paper. It is completely analogous to the classical construction (see Example 1.1 of \cite{GJ} on page 3).
\begin{defn}\label{2.4} As usual let $\mathbf{\Delta}$ denote the category whose objects are finite ordinal numbers $\{\underline n\in\omega\}$, and morphism are order-preserving maps. Then there is a functor
$$\mathbf{\Delta}\longrightarrow\mathcal{LSA}_K,
\quad\underline n\mapsto\Delta^n_K,$$
where $\Delta^n_K$ is the semi-algebraic set mentioned already in the introduction:
$$\Delta^n_K=\{(x_0,x_1,\ldots,x_n)\in K^n\mid \sum_{i=0}^nx_i=1\textrm{ and }0\leq x_i\textrm{ }(\forall i=0,1,\ldots,n)\},$$
where $\leq$ denotes the canonical ordering of $K$, and for every morphism $\phi:\underline n\to\underline m$ of $\mathbf{\Delta}$ the functor assigns the semi-algebraic map $|\phi|:\Delta^n_K\to\Delta^m_K$ given by the rule:
$$|\phi|(t_0,\ldots,t_n)=(s_0,\ldots,s_n),$$
where
$$s_i=\begin{cases}
    0 & \text{, if } \phi^{-1}(i)=\emptyset,\\
    \sum_{l\in\phi^{-1}(i)}t_j & \text{, otherwise. } \end{cases}$$
One needs to repeat a standard argument  to verify that this construction is indeed a functor.
\end{defn}
\begin{defn}\label{2.5} Let $SSets$ denote the category of simplicial sets, and let Ho$(SSets)$ denote the usual homotopy category of $SSets$. The singular set $S_*(M)$ of a locally semi-algebraic space $M$ is the simplicial set given by the correspondence:
$$\underline n\mapsto Hom_{\mathcal{LSA}}(\Delta_K^n,M),
$$
where for every morphism $\phi:\underline n\to\underline m$ of $\mathbf{\Delta}$ the corresponding map
$$Hom_{\mathcal{LSA}}(\Delta_K^m,M)\longrightarrow
Hom_{\mathcal{LSA}}(\Delta_K^n,M)$$
is given by the rule $f\mapsto f\circ|\phi|$. We will call $S_*(M)$ the simplicial set of semi-algebraic singular simplexes of $M$, and also the topological type of $M$. Clearly the correspondence $M\mapsto S_*(M)$ is a functor $\mathcal{LSA}\to SSets$. We define the homotopy type of $M$ as the weak equivalence class of $S_*(M)$ in Ho$(SSets)$.
\end{defn}
\begin{lemma}\label{2.6} The topological type of every object of $\mathcal{LSA}$ is fibrant. 
\end{lemma}
\begin{proof} We define the $k$-th horn $\Lambda^n_{k,K}\subset\Delta^n_K$ as the following semi-algebraic set:
$$\Lambda^n_{k,K}=\{(x_0,x_1,\ldots,x_n)\in\Delta^n_K \mid x_i=0\textrm{ for some $i\neq k$ and }i\in\{0,1,\ldots,n\}\}.$$
By definition we have to check that for every morphism $f:\Lambda^n_{k,K}\to M$ of locally semi-algebraic spaces there is a morphism $g:\Delta^n_K\to M$ in $\mathcal{LSA}$ such that $f=g\circ i^n_k$, where $i^n_k$ is the inclusion map:
$$\xymatrix{ \Lambda^n_{k,K}\ar[r]^{f} \ar@{^{(}->}[d]_{i^n_k}& M \\
 \Delta^n_K\ar@{.>}[ru]_{g}. & }$$
This claim follows from the fact that there is a semi-algebraic map $\pi^n_k:\Delta^n_K
\to \Lambda^n_{k,K}$ such that $\pi^n_k\circ i^n_k=\textrm{id}_{\Lambda^n_{k,K}}$, for example the projection from the point $\mathbf c=(c_0,\ldots,c_n)$, where
$$c_i=\begin{cases}
2 & \text{, if } i=k,\\
\frac{-1}{n} & \text{, otherwise. } \end{cases}$$
In fact here is an explicit formula: for every $i\neq k, i\in\{0,1,\ldots,n\}$ let
$\Delta^n_K[i,k]\subset\Delta^n_K$ be the following semi-algebraic set:
$$\Delta^n_K[i,k]=\{(x_0,x_1,\ldots,x_n)\in\Delta^n_K \mid
x_i\leq x_j\textrm{ for every $j\neq k$ and }j\in\{0,1,\ldots,n\}\}.$$
These cover $\Delta^n_K$. Then the $j$-th coordinate $y_j$ of
$\pi^n_k(x_0,x_1,\ldots,x_n)$ for an element $(x_0,x_1,\ldots,x_n)\in\Delta^n_K[i,k]$ is given by:
$$y_j=\begin{cases}
\frac{x_j-x_i}{nx_i+1} & \text{, if } j\neq k,\\
\frac{x_k+2nx_i}{nx_i+1} & \text{, otherwise. } \end{cases}$$
\end{proof}
\begin{defn} Let $SSets^0$ denote the category of pointed simplicial sets, and let Ho$(SSets^0)$ denote the usual pointed homotopy category of $SSets^0$. A pointed locally semi-algebraic space is a locally semi-algebraic space $M$ and a point $m\in M$. These form a category $\mathcal{LSA}^0_K$ in an obvious way: a morphism $(M,m)\to(N,n)$ is a morphism $\phi:M\to N$ of locally semi-algebraic spaces such that the map on underlying topological spaces takes $m$ to $n$. For every locally semi-algebraic space $M$ and $m\in M$ let $s_m$ denote the unique morphism $\Delta^0_K\to M$ whose image is $m$. Then there is a unique functor $\mathcal{LSA}^0_K\to SSets^0$ which assigns $(S_*(M),s_m)$ to $(M,m)$ for every object $(M,m)$ of $\mathcal{LSA}^0_K$. We will call the latter the pointed topological type of pointed locally semi-algebraic spaces.
\end{defn}
\begin{defn} Let $V\subset K^m,W\subset K^n$ be two semi-algebraic sets. Then the semi-algebraic set $V\times W\subset K^m\times K^n=K^{m+n}$ with the projection maps
$\pi_1:V\times W\to V,\pi_2:V\times W\to W$ is the product of $V$ and $W$ in
$\mathcal{SA}_K$ in the categorical sense. Using this fact in Example 2.5 of \cite{DeKn2} on pages 14--15 Delf and Knebusch constructs for every pair of locally semi-algebraic spaces $M,N$ a product $M\times N$ in the category $\mathcal{LSA}_K$. Let
$I_K=[0,1]=\{a\in K\mid 0\leq a\leq1\}$ be the analogue of the unit interval. A homotopy $H$ between two maps $f:M\to N,g:M\to N$ of locally semi-algebraic spaces is a map $H:M\times I_K\to N$ in $\mathcal{LSA}_K$ such that $H|_{M\times\{0\}}=f$ and $H|_{M\times\{1\}}=g$. We say that $f$ and $g$ are homotopy-equivalent if there is a homotopy between them. 
\end{defn}
\begin{defn} A locally semi-algebraic pair is a locally semi-algebraic space $M$ and a subset $M_0\subseteq M$. These form a category $\mathcal{LSA}^{p}_K$ in an obvious way: a morphism $(M,M_0)\to(N,N_0)$ is a morphism $\phi:M\to N$ of locally semi-algebraic spaces such that the map on underlying topological spaces takes $M_0$ to $N_0$. Clearly
$\mathcal{LSA}^0_K$ is a full subcategory, and there is also a natural functor $\mathcal{LSA}_K\to \mathcal{LSA}^p_K$ mapping every $lsa$-space $M$ to the pair $(M,\emptyset)$. A homotopy $H$ of two morphisms $f:(M,M_0)\to(N,N_0)$ and  $g:(M,M_0)\to(N,N_0)$ in $\mathcal{LSA}_K^p$ is a map $H:M\times I_K\to N$ in $\mathcal{LSA}_K$ such that $H|_{M\times\{0\}}=f,H|_{M\times\{1\}}=g$ and $H$ maps $M_0\times I_K$ to $N_0$, and we say that $f$ and $g$ are homotopy-equivalent if there is a homotopy between them. Homotopy equivalence is an equivalence relation; for every morphism $f$ in $\mathcal{LSA}_K^p$ let $[f]$ denote its homotopy equivalence class. We will call such a homotopy between two morphisms in $\mathcal{LSA}_K^0$ a pointed homotopy, and two morphisms in
$\mathcal{LSA}_K^0$ are pointed homotopy-equivalent if there is a pointed homotopy between them. 
\end{defn}
\begin{lemma}\label{2.10} $(i)$ Let $f:M\to N,g:M\to N$ be two homotopy-equivalent maps of locally semi-algebraic spaces. Then the induced maps $S_*(f):S_*(M)\to S_*(N)$ and $S_*(g):S_*(M)\to S_*(N)$ are homotopy-equivalent in $SSets$. 

$(ii)$ For every pair $f:(M,m)\to(N,n)$ and $g:(M,m)\to(N,n)$ of two homotopy-equivalent maps of pointed locally semi-algebraic spaces the induced maps $S_*(f):S_*(M,s_m)\to S_*(N,s_n)$ and $S_*(g):S_*(M,s_m)\to S_*(N,s_n)$ are homotopy-equivalent in $SSets^0$.
\end{lemma}
\begin{proof} Recall that for every $n\in\mathbb N$ the standard $n$-simplex $\Delta ^n$ is the simplicial set which as a functor $\mathbf{\Delta}^{op}\to Sets$  is $\textrm{Hom}_{\mathbf{\Delta}}(\cdot,\underline n)$. Note that every map $f:\Delta^n_K\to P$ of $lsa$-spaces furnishes a map $\Delta ^n\to S_*(P)$ of simplicial sets as follows: to every $\phi\in\textrm{Hom}_{\mathbf{\Delta}}(\underline m,\underline n)$ we assign $f\circ|\phi|\in S_m(P)$. In particular the semi-algebraic map:
$$\Delta^1_K\to I_K,\quad (x_0,x_1)\mapsto x_0$$
furnishes a map $\iota_1:\Delta^1\to S_*(I_K)$ in $SSets$.

Now let $f:M\to N$ and $g:M\to N$ be two maps of locally semi-algebraic spaces and let $H:M\times I_K\to N$ be a homotopy between them. Since the functor of semi-algebraic singular simplexes commutes with products we get that $S_*(M\times I_K)=S_*(M)\times S_*(I_K)$. Then the composition of
$$\textrm{id}_{S_*(M)}\times\iota_1:S_*(M)\times
\Delta^1\longrightarrow S_*(M)\times S_*(I_K)$$
and
$$S_*(H):S_*(M)\times S_*(I_K)=S_*(M\times I_K)\longrightarrow S_*(N)$$
and is a homotopy:
$$S_*(H)\circ(\textrm{id}_{S_*(M)}\times\iota_1):S_*(M)\times\Delta^1\longrightarrow S_*(N)$$
between $S_*(f)$ and $S_*(g)$ in $SSets$. If $f$ and $g$ are pointed maps
$(M,m)\to(N,n)$ and $H$ is a pointed homotopy between $S_*(f)$ and $S_*(g)$ in $SSets^0$.
\end{proof}
\begin{rem}\label{2.13a} Note that both $X=\Delta^n_K$ and
$\Lambda^n_{k,K}$ are contractible, that is, there is a homotopy
$$H:X\times I\longrightarrow X$$
such that $H|_{X\times\{0\}}=\textrm{id}_X$ and $H|_{X\times\{1\}}$ is constant. For example for every $k\in\{0,1,\ldots,n\}$ the map
$H:\Delta^n_K\times I\to\Delta^n_K$ given by
$$((x_0,x_1,\ldots,x_n),\lambda)\mapsto
((1-\lambda)x_0,(1-\lambda)x_1,\ldots,
1-(1-\lambda)(1-x_k),\ldots,
(1-\lambda)x_n)$$
and its restriction onto $\Lambda^n_{k,K}$ will do. Therefore by Lemma \ref{2.10} above the simplicial sets $S_*(\Delta^n_K)$ and
$S_*(\Lambda^n_{k,K})$ are contractible, too.
\end{rem}
\begin{defn} We define the boundary $\partial\Delta^n_K\subset\Delta^n_K$ of the $n$-dimensional simplex over $K$ as the following semi-algebraic set:
$$\partial\Delta^n_K=\{(x_0,x_1,\ldots,x_n)\in\Delta^n_K \mid \textrm{$x_i=0$ for some }i\in\{0,1,\ldots,n\}\}.$$
Moreover we define the $n$-dimensional cube $I_K^n$ and its boundary $\partial I_K^n$ (over $K$) as the following semi-algebraic sets:
$$I^n_K=\underbrace{I_K\times I_K\times\cdots\times I_K}_n=
\{(x_1,\ldots,x_n)\in\Delta^n_K \mid 0\leq x_i\leq1\ (\forall i\in\{1,\ldots,n\})\},$$
and
$$\partial I^n_K=\{(x_1,\ldots,x_n)\in I^n_K \mid \textrm{$x_i=0$ or $x_i=1$ for some }i\in\{1,\ldots,n\}\},$$
respectively. 
\end{defn}
\begin{lemma}\label{2.12} There is an isomorphism
$$\lambda_n:(\Delta^n_K,\partial\Delta^n_K)\to(I^n_K,\partial I^n_K)$$
in $\mathcal{LSA}^p_K$. 
\end{lemma}
\begin{proof} We are going to prove the claim by induction on $n$. For $n=1$ the map:
$$\lambda_1:\Delta^1_K\to I^1_K,\quad\lambda(x_0,x_1)=
x_0$$
will do. Now assume that $n\geq2$ and $\lambda_{n-1}$ has already been constructed. Let $T^n_K$ denote the semi-algebraic set:
$$T^n_K=\{(x_1,\ldots,x_n)\in K^n \mid 0\leq x_1\leq1,\ 
\frac{x_1}{2}\leq x_i\leq1-\frac{x_1}{2},\ \forall i\in\{2,\ldots,n\}\}$$
over $K$, and let $\partial T^n_K$ be its boundary:
$$\partial T^n_K=\{(x_1,\ldots,x_n)\in T^n_K \mid
x_1(1-x_1)\cdot\prod_{i\in\{2,\ldots,n\}}
(x_i-\frac{x_1}{2})(x_i-1+\frac{x_1}{2})=0\},$$
and let $T^n_{0,K}$ be its base:
$$T^n_{0,K}=\{(x_1,\ldots,x_n)\in T^n_K \mid 0=x_1\}.$$
Also set:
$$\Delta^n_{0,K}=\{(x_1,\ldots,x_n)\in \Delta^n_{0,K}
\mid 0=x_1\}.$$
It will be sufficient to show that there are isomorphisms
$$\alpha_n\!:\!(\Delta^n_K,\partial\Delta^n_K)
\to (T^n_K,\partial T^n_K),
\ \beta_n\!:\!(T^n_K,\partial T^n_K)\to(I^n_K,\partial I^n_K)$$
in $\mathcal{LSA}^p_K$. We will first construct $\alpha_n$. Since
$$T^n_{0,K}=\{(x_1,\ldots,x_n)\in K^n \mid
\textrm{$0=x_1$ and $(x_2,\ldots,x_n)\in I^{n-1}_K$}\},$$
$$\Delta^n_{0,K}=\{(x_1,\ldots,x_n)\in K^n \mid
\textrm{$0=x_1$ and $(x_2,\ldots,x_n)\in \Delta^{n-1}_{K}$}\},$$
there is an isomorphism $\lambda_{n-1}:\Delta^n_{0,K}\to
T^n_{0,K}$ by the induction hypothesis. We take $\alpha_n$ to be the cone of $\lambda_{n-1}$ over the bases $\Delta^n_{0,K}$ and $T^n_{0,K}$. More explicitly set:
$$\mathbf v=(1,0,\ldots,0)\in\Delta^n_K,\quad
\mathbf w=(1,\frac{1}{2},\ldots,\frac{1}{2})\in T^n_K.$$
Then we set $\alpha_n(\mathbf v)=\mathbf w$, while for every $\mathbf v\neq\underline x=(x_1,\ldots,x_n)\in C^n_K$ we have:
$$\alpha_n(\underline x)=\mathbf w+s(\underline x)^{-1}\cdot\left(
\lambda_{n-1}(\mathbf v+s(\underline x)\cdot(\underline x-\mathbf v))
-\mathbf w\right),$$
where $s(\underline x)$ is the unique element of $K$ such that
$s(\underline x)\geq1$ and $\mathbf v+s(\underline x)\cdot(\underline x-\mathbf v)\in\Delta^n_{0,K}$. Clearly $\alpha_n$ is a definable map. Since we can write down first order formulas in the language of ordered fields stating that $\alpha_n$ is continuous and a bijection, by the completeness of the theory of real closed fields we get that we only need to verify that $\alpha_n$ is continuous and a bijection over the real numbers, where this is well-known. 

The construction of $\beta_n$ is similar; informally we blow $T^n_K$ onto $I^n_K$ from the point $\mathbf c=(0,\frac{1}{2},\ldots,\frac{1}{2})\in T^n_K$. More explicitly we have
$\beta_n(\underline x)=\underline x$ for every $\underline x=(x_1,\ldots,x_n)\in T^n_{0,K}$, while for every $\underline x=(x_1,\ldots,x_n)\in T^n_K-T^n_{0,K}$ we have:
$$\beta_n(\underline x)=\frac{s_u(\underline x)}{s_l(\underline x)}\cdot(\underline x-\mathbf c)+\mathbf c,$$
where $s_l(\underline x),s_u(\underline x)$ are the unique positive elements of $K$ such that $s_l(\underline x)\cdot(\underline x-\mathbf c)+\mathbf c\in\partial T^n_K$ and $s_u(\underline x)\cdot(\underline x-\mathbf c)+\mathbf c\in I^n_K$. We may argue as above to see that $\beta_n$ is a semi-algebraic map which is an isomorphism between the pairs
$(T^n_K,\partial T^n_K)$ and $(I^n_K,\partial I^n_K)$.
\end{proof}
\begin{defn}\label{2.13} Let $M$ be a pointed $lsa$-space. We say that two points $x,y\in M$ are path-connected if there is a map $\phi:I_K\to M$ of $lsa$-spaces such that $\phi(0)=x$ and $\phi(1)=y$. This is an equivalence relation, and we will denote by $\pi_0(M)$ its equivalence classes. Now let $(M,m)$ be a pointed $lsa$-space. For every positive integer $n$ let $\pi_n(M,m)$ denote the homotopy classes of maps in
$\mathcal{LSA}^{p}_K$:
$$(\Delta_K^n,\partial\Delta_K^n)\longrightarrow (M,m).$$
Let $\pi^n_k:\Delta^n_K\to \Lambda^n_{k,K}$ be the semi-algebraic map constructed in the proof of Lemma \ref{2.6}. Given two such maps $f,g:(\Delta_K^n,\partial\Delta_K^n)\longrightarrow (M,m)$ we define their composition $f*g$ as the composition of the inclusion $\Delta^n_K\to\Delta^{n+1}_K$ given by the rule
$$(x_0,x_1,\ldots,x_n)\mapsto(x_0,x_1,\ldots,x_n,0),$$
the map $\pi^{n+1}_{n+1}:\Delta^{n+1}_K\to
\Lambda^{n+1}_{n+1,K}$, and the map $\Lambda^{n+1}_{n+1,K}\to M$ given by the rule:
$$f*g(x_0,x_1,x_2,\ldots,x_{n+1})=\begin{cases}
f(x_1,x_2,\ldots,x_{n+1}) & \text{, if } x_0=0,\\
g(x_0,x_2,\ldots,x_{n+1})  & \text{, if } x_1=0,\\
m & \text{, otherwise. } \end{cases}$$
This map is well-defined, semi-algebraic and its homotopy class
$[f*g]\in\pi_n(M,m)$ only depends on the homotopy classes $[f],[g]\in\pi_n(M,m)$ and hence this composition induces a binary operation on $\pi_n(M,m)$ which makes $\pi_n(M,m)$ into a group. All these can be seen directly but also follow from the lemma below.
\end{defn}
\begin{lemma}\label{2.14} $(i)$ For every $lsa$-space $M$ there is a natural bijection $\kappa_0^M:\pi_0(M)\to\pi_0(S_*(M))$. 

$(ii)$ Let $(M,m)$ be a pointed $lsa$-space. Then there is a natural group isomorphism $\kappa_i^{M,m}:\pi_i(M,m)\to\pi_i(S_*(M),s_m)$ for every positive integer $i$. 
\end{lemma}
\begin{proof} Since $S_*(M)$ is fibrant by Lemma \ref{2.6}, claim $(i)$ is immediate from the definitions. Part $(ii)$ could be proved the  same way as Lemma \ref{2.10}. We leave the details to the reader. 
\end{proof}
\begin{defn} Let $L$ be a real closed extension of $K$. Let $V\subset K$ be a semi-algebraic set over $K$ defined by a formula $\phi$. We define the extension $V_L\subset L^n$ as the semi-algebraic set defined by the formula $\phi$ over $L$. By model completeness $V_L$ is independent of the choice of $\phi$. Similarly we define the extension $f_L:V_L\to W_L$ of a semi-algebriac map $f:V\to W$ as the unique semi-algebraic map over $L$ whose graph $\Gamma_{f_L}\subset V_L\times W_L$ is the extension of the graph
$\Gamma_f\subset V\times W$ of $f$. This way we get a functor $e_{L/K}:\mathcal{SA}_K\to\mathcal{SA}_L$ which we will call the extension functor. In Example 2.10 of \cite{DeKn2} on page 19 a functor $e_{L/K}:\mathcal{LSA}_K\to
\mathcal{LSA}_L$ is constructed extending (or generalising) this one which commutes with direct limits. By the usual abuse of language we also call this the extension functor from $K$ to $L$.
\end{defn}
\begin{defn} By slight abuse of notation let $e_{L/K}$ also denote the unique functor $\mathcal{LSA}^0_K\to\mathcal{LSA}^0_L$ such that for every object $(M,m)$ of $\mathcal{LSA}^0_K$ we have $e_{L/K}(M,m)=(e_{L/K}(M),m_{L/K})$, where $m_{L/K}\in e_{L/K}(M)$ is the unique point such that $e_{L/K}(s_m)=s_{m_{L/K}}$.  By the functoriality of $e_{L/K}$ there is a natural map $b^M_{L/K}:S_*(M)\to S_*(e_{L/K}(M))$ of simplicial sets for every object $M$ of $\mathcal{LSA}_K$, and similarly there is a natural map $b^{M,m}_{L/K}:S_*(M,s_m)\to S_*(e_{L/K}(M),s_{m_{L/K}})$ of pointed simplicial sets for every object
$(M,m)$ of $\mathcal{LSA}^0_K$.
\end{defn}
\begin{defn} The strong topology $\mathfrak{T}(M)$ on a locally semi-algebraic space $M$ is the topology generated by the open subsets of the Grothendieck topology $\mathfrak{S}(M)$ in the usual sense (see the paragraph after Proposition 3.11 of \cite{DeKn2} on page 31). This construction is functorial. For every $M$ as above let $M^{top}$ denote the topological space $(M,\mathfrak{T}(M))$ and let $S^{top}_*(M)$ denote the singular simplicial set of $M^{top}$. When $K$ is the real number field $\mathbb R$ there is a natural simplicial map $t^M:S_*(M)\to S^{top}_*(M)$ since every semi-algebraic map is continuous with respect to the strong topology. Similarly for a pointed locally semi-algebraic space $(M,m)$ the pointed simplicial set $(S^{top}_*(M),s_m)$ is well-defined and when $K=\mathbb R$ there is a natural simplicial map $t^{M,m}:(S_*(M),s_m)\to(S^{top}_*(M),s_m)$.
\end{defn}
\begin{defn} Next we are going to recall the natural analogues of some basic notions of general topology in our setting. A locally semi-algebraic subspace $N$ of a $lsa$-space $M$ is a subset $N\subset M$ such that for some admissible covering $N=\bigcup_{\alpha}N_{\alpha}$ by $lsa$-spaces associated to semi-algebraic sets the intersection $N\cap M_{\alpha}$ is a semi-algebraic subset for each $\alpha$. Following Definition 1 of section 4 of \cite{DeKn2} on page 42 we will say that a locally semi-algebraic space $M$ is regular if for every locally semi-algebraic subset $A$ of $M$ closed with respect to the strong topology and every point $x$ in $M-A$ there exists sets $U,V\subset M$ open with respect to the strong topology such that $x\in U$, $A\subset V$ and $U\cap V=\emptyset$. After Definition 2 of section 4 of \cite{DeKn2} on page 43 we will say that a locally semi-algebraic space $M$ is
paracompact if there is a locally finite covering admissible $M=\bigcup_{\alpha}M_{\alpha}$ by $lsa$-spaces associated to semi-algebraic sets. (Recall that a covering $M=\bigcup_{\alpha}M_{\alpha}$ is locally finite if for every $\alpha$ the set $\{\beta\mid M_{\alpha}\cap M_{\beta}\neq\emptyset\}$ is finite.)
\end{defn}
\begin{thm}\label{basic_comparison} (i) Let $L$ be a real closed extension of $K$. For every regular, paracompact $lsa$-space $M$ over $K$ the natural map $b^M_{L/K}:S_*(M)\to S_*(e_{L/K}(M))$ is a weak equivalence in $SSets$. Moreover for every $m\in M$ the natural map $b^{M,m}_{L/K}:S_*(M,s_m)\to S_*(e_{L/K}(M),s_{m_{L/K}})$ is a weak equivalence in $SSets^0$.

(ii) Assume that $K$ is the real number field $\mathbb R$. For every regular, paracompact $lsa$-space $M$ over $K$ the natural map $t^M:S_*(M)\to S^{top}_*(M)$ is a weak equivalence in $SSets$. Moreover for every $m\in M$ the natural map $t^{M,m}:S_*(M,s_m)\to S^{top}_*(M,s_m)$ is a weak equivalence in $SSets^0$.
\end{thm}
\begin{proof} Let $K$, $L$ and $M$ be the same as in part $(i)$. Since $S_*(M)$ and $S_*(e_{L/K}(M))$ are fibrant by Lemma \ref{2.6}, it will be sufficient to prove that the map $\pi_0(b^M_{L/K}):\pi_0(S_*(M))\to
\pi_0(S_*(e_{L/K}(M)))$ induced by $b^M_{L/K}$ and for every $m\in M$  and positive integer $i$  the homomorphism $\pi_i(b^{M,m}_{L/K}):\pi_i(S_*(M),s_m)\to\pi_i(S_*(e_{L/K}(M)),s_{m_{L/K}})$ induced by $b^{M,m}_{L/K}$ are bijections.  The functor $e_{L/K}$ induces a map $\pi_{0,L/K}^M:\pi_0(M)\to\pi_0(e_{L/K}(M))$ of sets and for every $m\in M$  and positive integer $i$  a homomorphism $\pi^{M,m}_{i,L/K}:\pi_i(M,m)\to\pi_i(e_{L/K}(M),m_{L/K})$ of groups. We claim that these maps are isomorphisms. Indeed they are obviously homomorphisms, and they are bijections by Lemma \ref{2.12} and by Theorem 6.3 of \cite{DeKn2} on pages 270-271. Since the diagrams:
$$\!\!
\xymatrix{\pi_0(M)\ar[r]^-{\pi_{0,L/K}^M}\ar[d]_{\kappa_0^M} & \pi_0(e_{L/K}(M))\ar[d]^{\kappa_0^{e_{L/K}(M)}} \\
\pi_0(S_*(M))\ar[r]^-{\pi_0(b^M_{L/K})} &  \pi_0(S_*(e_{L/K}(M))}
\ \ 
\xymatrix{\pi_i(M,m)\ar[r]^-{\pi^{M,m}_{i,L/K}}\ar[d]_{\kappa_i^{M,m}} & \pi_i(e_{L/K}(M),m_{L/K})\ar[d]^{\kappa^{e_{L/K}(M),m_{L/K}}_i} \\
\pi_i(S_*(M),s_m)\ar[r]^-{\pi_i(b^{M,m}_{L/K})} &
\pi_i(S_*(e_{L/K}(M),s_{m_{L/K}})}$$
are commutative, claim $(i)$ follows from Lemma \ref{2.14}.

The proof of claim $(ii)$ is very similar: for every such $M$ it will be sufficient to prove that the map $\pi_0(t^M):\pi_0(S_*(M))\to
\pi_0(S^{top}_*(M))$ induced by $t^M$ and for every $m\in M$ and positive integer $i$ the homomorphism $\pi_i(t^{M,m}):\pi_i(S_*(M),m)\to\pi_i(S^{top}_*(M),m)$ induced by $t^{M,m}$ are bijections. There are natural maps
$\tau^M_0:\pi_0(M)\to\pi_0(M^{top})$ of sets and for every $m\in M$  and positive integer $i$ a homomorphism $\tau^{M,m}_i:\pi_i(M,m)\to\pi_0(M^{top},m)$ of groups, which are obviously homomorphisms, and which are also bijections by Lemma \ref{2.12} and Theorem 6.4 of \cite{DeKn2} on page 271. Since the diagrams:
$$\!\!
\xymatrix{\pi_0(M)\ar[r]^-{\tau_0^M}\ar[d]_{\kappa_0^M} & \pi_0(M^{top})\ar[d]^{\lambda_0^M} \\
\pi_0(S_*(M))\ar[r]^-{\pi_0(t^M)} &  \pi_0(S^{top}_*(M))}
\ \ 
\xymatrix{\pi_i(M,m)\ar[r]^-{\tau^{M,m}_i}\ar[d]_{\kappa_i^{M,m}} & \pi_i(M^{top},m)\ar[d]^{\lambda^{M,m}_i} \\
\pi_i(S_*(M),s_m)\ar[r]^-{\pi_i(t^{M,m})} &
\pi_i(S^{top}_*(M),s_m)}$$
are commutative, where the left vertical arrows $\lambda^M_0$ and
$\lambda^{M,m}_i$ are the usual comparison maps between the topological and simplicial connected components and homotopy groups, respectively. Since these maps are bijections, claim $(ii)$ follows from from Lemma \ref{2.14}.
\end{proof}

\section{Closed simplicial complexes}

\begin{defn} A closed simplicial complex is a pair $A=(A_v,A_s)$ such that
\begin{enumerate}
\item[$(i)$] each element of $A_s$ is a finite subset of $A_v$,
\item[$(ii)$] the union of $A_s$ is $A_v$, 
\item[$(iii)$] if $B$ is a subset of an element of $A_s$ then $B\in A_s$.
\end{enumerate}
Note that these objects are simply called simplicial complexes in homotopy theory, but we will follow this terminology used in real algebraic geometry, in order to avoid confusion. As usual we will call $A_v$ the set of vertices of $A$, and $A_s$ the set of simplexes of $A$, respectively. We say that $A$ is finite if $A_v$ is. Closed simplicial complexes form a category
$\mathcal{SC}$ where a morphism between two objects $A=(A_v,A_s)$ and $B=(B_v,B_s)$ is a map $f:A_v\to B_v$ such that $f(C)\in B_s$ for every $C\in A_s$. There is a functor $A\mapsto A_*$ from $\mathcal{SC}$ to $SSets$ which is defined as follows:
$$A_n=\{\phi\in\textrm{Hom}_{Sets}(\underline n,A_v)\mid
\textrm{Im}(\phi)\in A_s\}$$
and for every $\psi\in\textrm{Hom}_{\mathbf{\Delta}}(\underline n,\underline m)$ the corresponding map $A_m\to A_n$ is $\phi\mapsto\psi\circ\phi$. This functor is fully faithful, and with its aid we will consider $\mathcal{SC}$ as a full subcategory of $SSets$. 
\end{defn}
\begin{defn}\label{3.5} Let $\mathbf v_0,\mathbf v_1,\ldots,\mathbf v_k\in K^n$. The convex hull of a $k$-tuple $\mathbf v_0,\ldots,\mathbf v_k\in K^n$ is the image of the map
$$\Delta^k_K\to K^n,\quad(x_0,x_1,\ldots,x_k)\mapsto
\sum_{j=0}^kx_j\mathbf v_k.$$
If $C$ is the finite set $\{\mathbf v_0,\mathbf v_1,\ldots,\mathbf v_k\}$, we let $[C]$ denote the convex hull of $\mathbf v_0,\ldots,\mathbf v_k$. Now let $A=(A_v,A_s)$ be a finite closed simplicial complex. Let $K^{A_v}$ be the finite dimensional $K$-linear vector space spanned formally by $A_v$. The realisation $|A|_K$ of $A$ over $K$ is the union of the convex hulls $[C]$ for every $C\in A_s$. As the union of finitely many semi-algebraic sets, the subset $|A|_K\subset K^{A_v}$ is also semi-algebraic. 
\end{defn}
\begin{lemma}\label{piecewise} For every finite closed simplicial complex $A$ and lsa-space $M$ a map $f:|A|_K\to M$ is semi-algebraic if and only if for every $C\in A_s$ the restriction $f|_{[C]}$ is semi-algebraic. 
\end{lemma}
\begin{proof} Since the composition of semi-algebraic maps is semi-algebraic, the first condition implies the second, so we only need to show the converse. Let $Z\subseteq M$ be a closed subset. Then $f|^{-1}_{[C]}(Z)$ is closed in $[C]$ by assumption for every $C\in A_s$. But $[C]$ is closed in $|A|_K$, and hence $f|^{-1}_{[C]}(Z)$ is closed in $|A|_K$, too. Therefore the finite union
$$\bigcup_{C\in A_s}f|^{-1}_{[C]}(Z)=
\bigcup_{C\in A_s}f^{-1}(Z)\cap[C]=f^{-1}(Z)$$
is closed, too. We get that $f$ is continuous. Now let $\{M_{\alpha}\}_{\alpha\in I}$ be an admissible covering of $M$ such that for every $\alpha\in I$ the ringed space
$(M_{\alpha},\mathcal O_M|_{M_{\alpha}})$ is a semi-algebraic space. For every $C\in A_s$ the covering
$\{f^{-1}(M_{\alpha})\cap[C]\}_{\alpha\in I}$ of $[C]$ is admissible, so there is a finite subset $I(C)\subseteq I$ such that $\{f^{-1}(M_{\alpha})\cap[C]\}_{\alpha\in I(C)}$ already covers $[C]$. Since $A_s$ is finite, we get that there is a finite subset $J\subseteq I$ such that $\{f^{-1}(M_{\alpha})\}_{\alpha\in J}$ already covers $|A|_K$. Each $f^{-1}(M_{\alpha})$ is open, as $f$ is continuous, and since it is a finite union of semi-algebraic sets:
$$f^{-1}(M_{\alpha})=\bigcup_{C\in A_s}f^{-1}(M_{\alpha})\cap[C],$$
it is semi-algebraic, too. Therefore $\{f^{-1}(M_{\alpha})\}_{\alpha\in I}$ is an admissible cover of $|A|_K$. Therefore it will be enough to check that for every $\alpha\in I$ the map
$$f|_{f^{-1}(M_{\alpha})}:f^{-1}(M_{\alpha})\to M_{\alpha}$$
is semi-algebraic. Since it is continuous, and $f^{-1}(M_{\alpha})$ is semi-algebraic, it will be enough to check that its graph is semi-algebraic. But it is the finite union of the graphs of the maps
$$f|_{f^{-1}(M_{\alpha})\cap[C]}:f^{-1}(M_{\alpha})\cap[C]\to M_{\alpha}\cap f([C]),$$
which are semi-algebraic.
\end{proof}
\begin{defn}\label{3.6} Given a morphism $f:A_v\to B_v$ between two finite closed simplicial complexes $A=(A_v,A_s)$ and $B=(B_v,B_s)$ we are going to construct a semi-algebraic map 
$|f|_K:|A|_K\to|B|_K$ as follows. Let $\mathbf w\in|A|_K$ be arbitrary. Let $C\in A_s$ be such that $\mathbf w\in[C]$. Write $C=\{\mathbf v_0,\mathbf v_1,\ldots,\mathbf v_k\}$. Then
$$\mathbf w=\sum_{j=0}^kx_j\mathbf v_j\textrm{ for some $(x_0,x_1,\ldots,x_k)\in\Delta^k_K$.}$$
We set:
\begin{equation}\label{3.6.1}
|f|_K(\mathbf w)=\sum_{j=0}^kx_jf(\mathbf v_j).
\end{equation}
\end{defn}
\begin{lemma} The map $|f|_K:|A|_K\to|B|_K$ is well-defined.
\end{lemma}
\begin{proof} For every $C\in A_s$ such that $\mathbf w\in[C]$ let
$|f|^{C}_K(\mathbf w)$ denote the right hand side of (\ref{3.6.1}) for the moment. We need to show that for every other $D\in A_s$ such that $\mathbf w\in[D]$ we have $|f|^{C}_K(\mathbf w)=
|f|^{D}_K(\mathbf w)$. Note that $C\cap D\in A_s$ so by symmetry it will be enough to show that $|f|^{C}_K(\mathbf w)=
|f|^{C\cap D}_K(\mathbf w)$. In other words we may assume without the loss of generality that $D\subseteq C$. Write $C=\{\mathbf v_0,\mathbf v_1,\ldots,\mathbf v_k\}$ and let $J\subseteq
\{0,1,\ldots,k\}$ be such that $D=\{\mathbf v_j\mid j\in J\}$. If 
$$\mathbf w=\sum_{j=0}^kx_j\mathbf v_k\textrm{ for some $(x_0,x_1,\ldots,x_k)\in\Delta^k_K$,}$$
then $x_j=0$ for every $j\not\in J$. Therefore
$$|f|^{C}_K(\mathbf w)=\sum_{j=0}^kx_jf(\mathbf v_j)=
\sum_{j\in J}^kx_jf(\mathbf v_j)=|f|^{D}_K(\mathbf w).$$
\end{proof}
It is clear from the definitions and Lemma \ref{piecewise} that the constructions of Definitions \ref{3.5} and \ref{3.6} above furnish a functor $|\cdot|_K:
\mathcal{SC}\to\mathcal{SA}_K$. 
\begin{prop}\label{3.8} Let $f:C\to D$ be a map of finite closed simplicial complexes. Then the following are equivalent:
\begin{enumerate}
\item[$(i)$] the map $f$ is a weak equivalence,
\item[$(ii)$] the map $|f|_K:|C|_K\to|D|_K$ is a weak equivalence for every real closed field $K$, 
\item[$(iii)$] the map $|f|_K:|C|_K\to|D|_K$ is a weak equivalence for a real closed field $K$.
\end{enumerate}
\end{prop}
\begin{proof} Let $L$ be a real closed extension of $K$. Then the image of $|f|_K:|C|_K\to|D|_K$ under the functor $e_{L/K}$ is $|f|_L:|C|_L\to|D|_L$. So the map $|f|_K$ is a weak equivalence if and only if $|f|_L$ is a weak equivalence by part $(i)$ of Theorem \ref{basic_comparison}. Now let $\mathbb Q^{real}$ denote the real algebraic closure of $\mathbb Q$ with respect to its unique ordering. Then for any real closed field $K$ as above there is an embedding of $\mathbb Q^{real}$ into $K$. By the above $|f|_K$ is a weak equivalence if and only if $|f|_{\mathbb Q^{real}}$ is a weak equivalence, so $(ii)$ and $(iii)$ are equivalent. Note that by definition $f$ is a weak equivalence if and only if the map
$S^{top}_*(|C|_{\mathbb R})\to S^{top}_*(|D|_{\mathbb R})$ induced by $|f|_{\mathbb R}$ is a weak equivalence. So by part $(ii)$ of Theorem \ref{basic_comparison} the morphism $f$ is a weak equivalence if and only if $|f|_{\mathbb R}:|C|_{\mathbb R}\to|D|_{\mathbb R}$ is a weak equivalence. So by the above $(i)$ and $(ii), (iii)$ are equivalent. 
\end{proof}
\begin{defn} Let $A=(A_v,A_s)$ be again a finite closed simplicial complex. We are going to define a map $\iota_A:A_*\to S_*(|A|_K)$ of simplicial sets as follows. For every $\psi:\underline n\to A_v$ in $A_n$ let
$\iota_A(\psi):\Delta^n_K\to|A|_K$ be the semi-algebraic map:
$$(x_0,x_1,\ldots,x_n)\mapsto\sum_{j=0}^nx_j\psi(j),\quad
\forall(x_0,x_1,\ldots,x_n)\in\Delta^k_K.$$
Since for every morphism $\phi:\underline m\to\underline n$ of
$\mathbf{\Delta}$ and $\psi$ as above we have $\iota_A(\psi\circ\phi)=\iota_A(\psi)\circ|\phi|$ (where $|\phi|$ is the semi-algebraic map introduced in Definition \ref{2.4}), the map $\iota_A$ is a cofibration of simplicial sets.
\end{defn}
\begin{prop}\label{3.10} For every finite closed simplicial complex $A$ and lsa-space $M$ the map:
\begin{equation}\label{3.10.1}
\beta_A:\textrm{\rm Hom}(|A|_K,M)\to\textrm{\rm Hom}(A_*,S_*(M))
\end{equation}
furnished by $f\mapsto S_*(f)\circ\iota_A$ is a bijection.
\end{prop}
\begin{proof} As usual for every natural number $n$ let
$\Delta^n$ denote the standard $n$-simplex in the category of simplicial sets (see Example 1.7 of \cite{GJ} on page 6). Moreover let $\mathbf n$ denote the closed simplicial complex whose set of vertices is $\underline n$ and whose set of simplexes are all subsets of $\underline n$. Then we have natural isomorphisms:
$$\Delta^n_K\cong|\mathbf n|_K
\quad\textrm{and}\quad
\Delta^n\cong\mathbf n_*$$
which in turn induce natural bijections:
$$a_n:\textrm{\rm Hom}(|\mathbf n|_K,M)\to S_n(M)
\quad\textrm{and}\quad
b_n:\textrm{\rm Hom}(\mathbf n_*,S_*(M))\to
S_n(M).$$
Since the diagram:
$$\xymatrix{
\textrm{\rm Hom}(|\mathbf n|_K,M)\ar[rd]^{a_n}
\ar[dd]^{\beta_{\mathbf n}} &  \\
& S_n(M) \\
\textrm{\rm Hom}(\mathbf n_*,S_*(M))\ar[ru]_{b_n} & }$$
is commutative, we get that the claim is true for $\mathbf n$. Now let us consider the general case. Note that images of the simplexes of $A_*$ under $\iota_A$ are sets of the form $[C]$ where $C$ runs through $A_s$. These sets cover $|A|_K$ so the map $\beta_A$ is injective. So we only have to show that this map is also surjective. Now let $g\in\textrm{\rm Hom}(A_*,S_*(M))$ be arbitrary. For every $C\in A_s$ let $C$ also denote by slight abuse of notation the subcomplex of $A$ whose set of vertices is $C$ and whose set of simplexes are all subsets of $C$. The inclusion $C\hookrightarrow A$ induces an identification of $|C|_K$ with $[C]$, and also a map
$$\textrm{\rm Hom}(A_*,S_*(M))\to
\textrm{\rm Hom}(C_*,S_*(M)).$$
Let $g_C$ denote the image of $g$ under this map. Since $C$ is isomorphic to $\mathbf n$, where $n$ is the cardinality of $C$, there is a unique $f_C\in\textrm{\rm Hom}(|C|_K,M)=
\textrm{\rm Hom}([C],M)$ whose image under $\beta_C$ is $g_C$ by the above. Now let $\mathbf w\in|A|_K$ be arbitrary, and let $C\in A_s$ be such that $\mathbf w\in[C]$. We set:
\begin{equation}\label{3.11.2}
f(\mathbf w)=f_C(\mathbf w).
\end{equation}
\begin{lemma} The map $f:|A|_K\to M$ is well-defined.
\end{lemma}
\begin{proof} We need to show that $f_{C}(\mathbf w)$ is independent of the choice of $C$. Let $D$ be another simplex in $A_s$. We claim that $f_C|_{[C]\cap[D]}=f_D|_{[C]\cap[D]}$. Note that $C\cap D\in A_s$ and
$[C]\cap [D]=[C\cap D]$ so by symmetry it will be enough to show that $f_C|_{[C\cap D]}=f_{C\cap D}$. In other words we may assume without the loss of generality that $D\subseteq C$. Since the diagram:
$$\xymatrix{ \textrm{\rm Hom}(|C|_K,M)
\ar[r]^{\beta_C}\ar[d] & \textrm{\rm Hom}(C_*,S_*(M))\ar[d] \\
\textrm{\rm Hom}(|D|_K,M)\ar[r]^{\beta_D} &
\textrm{\rm Hom}(D_*,S_*(M))}$$
induced by the inclusion map $D\hookrightarrow C$ is commutative, and its horizontal maps are bijections, the claim follows. 
\end{proof}
Note that the restriction of $f$ onto $[C]$ is semi-algebraic for every $C\in A_s$, since it is equal to $f_C$. Therefore by Lemma \ref{piecewise} the map $f$ is semi-algebraic. Clearly $g=S_*(f)\circ\iota_A$, so we get that $\beta_A$ is surjective, too.
\end{proof}

\section{The fibration theorem: locally trivial fibrations}\label{fibration_section1}

\begin{defn} Let $Y_*$ be a simplicial set and let $\mathcal U=\{U_*\mid U_*\subseteq Y_*\}$ be a collection  of simplicial subsets of $Y_*$. We say that a simplicial map $g:R_*\to Y_*$ is subordinate to $\mathcal U$ if for every $y\in R_*$ there is a $U_*\in\mathcal U$ such that $g(y)\in U_*$. The property of maps subordinate to $\mathcal U$ that we will use repeatedly is that if $h$ is such and $f:P_*\to R_*$ is a simplicial map then the composition $g\circ f$ is also subordinate to $\mathcal U$. Now let $p:X_*\to Y_*$ be a map of simplicial sets and let $\mathcal U$ be a collection of simplicial subsets of $Y_*$ as above. We say that a cofibration $i:A_*\to B_*$ has the $\mathcal U$-subordinate left lifting property with respect to $p:X_*\to Y_*$ if for every commutative diagram:
$$\xymatrix{ A_*\ar[r]^{f}\ar[d]_{i} & X_*\ar[d]^{p} \\
B_*\ar[r]^{g} \ar@{.>}[ru]^{h}& Y _*}$$
in $SSets$ such that $g:B\to Y_*$ is subordinate to $\mathcal U$ there is a morphism $h:B_*\to X_*$ which makes the diagram commutative. 
\end{defn}
\begin{lemma}\label{3.2} The class of cofibrations $i:A_*\to B_*$ which has the
$\mathcal U$-subordinate left lifting property with respect to $p:X_*\to Y_*$ is saturated. 
\end{lemma}
\begin{proof} Trivally all isomorphisms has the $\mathcal U$-subordinate left lifting property. Now consider the commutative diagram
$$\xymatrix{ A_*\ar[r]^{c}\ar[d]_{i} & C_*\ar[r]^{f}\ar[d]^{i_C} & X_*\ar[d]^{p} \\
B_*\ar[r]^{c_B} \ar@{.>}[rru]^{\!\!\!h}& B_*\cup_{A_*}C_* \ar[r]^{g}\ar@{.>}[ru]_{h_C} & Y _* }$$
in $SSets$ where the cofibration $i:A\to B$ has the $\mathcal U$-subordinate left lifting property with respect to $p:X_*\to Y_*$, the left square is a push-out diagram, and the map $g$ is subordinate to $\mathcal U$. Then $g\circ c_B$ is also subordinate to $\mathcal U$, so there is a map $h:B_*\to X_*$ which makes the diagram commutative. Let $h_C:B_*\cup_{A_*}C_*\to X_*$ be the push-out of $h$. By the universal property of push-outs this map makes the diagram commutative. We get that $i_C:C_*\to B_*\cup_{A_*}C_*$ has the
$\mathcal U$-subordinate left lifting property with respect to $p:X_*\to Y_*$.  Next consider the commutative diagram
$$\xymatrix{ A'_*\ar[r]^{a_0}\ar[d]_{i'}\ar@/^1pc/[rr]^{\textrm{id}_{A'_*}} & A_*\ar[r]^{a_1}\ar[d]_{i} & A'_*\ar[r]^{f}\ar[d]^{i'} & X_*\ar[d]^{p} \\
B'_*\ar[r]^{b_0} \ar@{.>}[rrru]^{h'}\ar@/_1pc/[rr]_{\textrm{id}_{B'_*}}& B_*\ar[r]^{b_1} \ar@{.>}[rru]^{\!\!\!h}& B'_* \ar[r]^{g} & Y _* }$$
in $SSets$ where the cofibration $i:A\to B$ has the $\mathcal U$-subordinate left lifting property with respect to $p:X_*\to Y_*$, and the map $g$ is subordinate to $\mathcal U$. Then $g\circ b_1$ is also subordinate to $\mathcal U$, so there is a map $h:B_*\to X_*$ which makes the diagram commutative. Let $h':B'_*
\to X_*$ be the composition $h\circ b_0$. This map will make the diagram above commutative, and therefore $i':A_*\to B_*$ also has the
$\mathcal U$-subordinate left lifting property with respect to $p:X_*\to Y_*$. Now let
$$
\xymatrix{ A^1_*\ar[r]^{i^1} & A^2_*\ar[r]^{i^2} & A^3_*\ar[r]^{i^3} &\cdots }
$$
be a diagram in $SSets$ such that all maps $i^k$ are cofibrations and have the $\mathcal U$-subordinate left lifting property with respect to $p:X_*\to Y_*$. For every positive integer $k$ let $i^{\infty}_k:A^k_*\to\varinjlim A^i_*$ be the limiting map and assume that we are given a diagram
\begin{equation}\label{3.2.1}
\xymatrix{ A^1_*\ar[r]^{f}\ar[d]_{i^{\infty}_1} & X_*\ar[d]^{p} \\
\varinjlim A^i_*\ar[r]^{g} \ar@{.>}[ru]^{h}& Y _*}
\end{equation}
in $SSets$ such that the map $g$ is subordinate to $\mathcal U$. We are going to construct a sequence of simplicial maps $h^k:A^k_*\to X_*$ for $k=1,2,3,\ldots$ such that $p\circ h^k=g\circ i^{\infty}_k$ and $h^k=i^k\circ h^{k+1}$ as follows. Set $h^1=f$. Clearly $p\circ h^1=g\circ i^{\infty}_1$. Assume now that the first $k$ maps $h^1,\ldots,h^k$ have been already constructed and consider the diagram:
$$\xymatrix{ A^k_*\ar[r]^{h^k}\ar[d]_{i^k} & X_*\ar[d]^{p} \\
A^{k+1}_*\ar[r]^{\ \ g\circ i^{\infty}_{k+1}} \ar@{.>}[ru]^{h^{k+1}}& Y _*.}$$
Because $p\circ h^k=g\circ i^{\infty}_k=g\circ i^{\infty}_{k+1}\circ i^k$ by the induction hypothesis, this diagram is commutative. Since $g\circ i^{\infty}_{k+1}:A^{k+1}_*\to Y _*$ is subordinate to $\mathcal U$ and $i^k:A^k\to A^{k+1}$ have the $\mathcal U$-subordinate left lifting property with respect to $p:X_*\to Y_*$, we get that there is a simplicial map $h^{k+1}:A^{k+1}_*\to X _*$ which makes the diagram commutative, so in particular $p\circ h^{k+1}=g\circ i^{\infty}_{k+1}$ and $h^k=i^k\circ h^{k+1}$. Because of the latter compatibility property there is a limit morphism $h=\varinjlim h^k$ which makes the diagram (\ref{3.2.1}) above commutative. So $i^{\infty}_1:A^1_*\to\varinjlim A^i_*$ have the $\mathcal U$-subordinate left lifting property with respect to $p:X_*\to Y_*$. 

Finally let $J$ be a index set, and for every $j\in J$ let $i^j:A^j_*\to B^j_*$ be a cofibration having the $\mathcal U$-subordinate left lifting property with respect to $p:X_*\to Y_*$. Assume that we are given a commutative diagram
\begin{equation}\label{3.2.2}
\xymatrix{\coprod_{j\in J}A^j_*\ar[r]^{f}\ar[d]_{\coprod_{j\in J}i^j} & X_*\ar[d]^{p} \\
\coprod_{j\in J}B^j_*\ar[r]^{g} \ar@{.>}[ru]^{h}& Y _*}
\end{equation}
in $SSets$ such that the map $g$ is subordinate to $\mathcal U$. By the universal property of coproducts there are simlicial maps $f^j:A^j_*\to X_*$ and $f^j:B^j_*\to Y_*$ for every $j\in J$ such that
$$f=\coprod_{j\in}f^j\textrm{ and }g=\coprod_{j\in}g^j,$$
and the diagram
\begin{equation}\label{3.2.3}
\xymatrix{A^j_*\ar[r]^{f^j}\ar[d]_{i^j} & X_*\ar[d]^{p} \\
B^j_*\ar[r]^{g^j} \ar@{.>}[ru]^{h^j}& Y _*}
\end{equation}
in $SSets$ is commutative for every $j\in J$. Because $g^j$ is the composition of the inclusion map $B^j\to\coprod_{j\in J}B^j_*$ and $g$, it is subordinate to $\mathcal U$, therefore there is a simplicial map $h^j:B^j_*\to X _*$ which makes the diagram (\ref{3.2.3}) commutative. Then the coproduct $h=\coprod_{j\in}h^j$ makes the diagram (\ref{3.2.2}) commutative, so $\coprod_{j\in J}i^j$ has the $\mathcal U$-subordinate left lifting property with respect to $p:X_*\to Y_*$, which concludes the proof of the lemma.
\end{proof}
Note that for every $U_*\in\mathcal U$ the pre-image $p^{-1}(U_*)\subseteq X_*$ is a simplicial subset of $X_*$.
\begin{cor}\label{3.3} Assume in addition that for every $U_*\in\mathcal U$ the restriction
$$p|_{p^{-1}(U_*)}:p^{-1}(U_*)\to U_*$$
is fibrant. Then the class of cofibrations $i:A_*\to B_*$ which has the $\mathcal U$-subordinate left lifting property with respect to $p:X_*\to Y_*$ contains all trivial cofibrations. 
\end{cor}
\begin{proof} Let $\Lambda^n_k\subset\Delta^n$ be the $k$-th horn in the standard $n$-dimensional simplex. Now consider the commutative diagram:
$$\xymatrix{ \Lambda^n_k\ar[r]^{f}\ar[d]_{i^n_k} & X_*\ar[d]^{p} \\
\Delta^n\ar[r]^{g} \ar@{.>}[ru]^{h}& Y _*}$$
where $i^n_k:\Lambda^n_k\to\Delta^n$ is the inclusion map and $g$ is subordinate to $\mathcal U$. This means that the unique non-degenerate simplex of $\Delta^n$ (corresponding to the identity map of $\underline n$) maps into $U_*$ for some $U_*\in\mathcal U$ with respect to $g$. This implies that $g$ maps $\Delta^n$ into $U_*$, and hence $f$ maps
$\Lambda^n_k$ into $p^{-1}(U_*)$. Since $p|_{p^{-1}(U_*)}:p^{-1}(U_*)\to U_*$ is fibrant, there is a morphism $h:\Delta^n\to X_*$, mapping into
$p^{-1}(U_*)$, which makes the diagram above commutative. We get that $i^n_k$ has the $\mathcal U$-subordinate left lifting property with respect to $p:X_*\to Y_*$. Since by the theory of anodyne extensions (see pages 14-19 and 62-63 in \cite{GJ}) the saturation of the inclusion maps $\Lambda^n_k\subset\Delta^n$ contains all trivial cofibrations, the claim follows from Lemma \ref{3.2}.
\end{proof}
\begin{defn} Let $(M,\mathfrak{S}(M),\mathcal O_M)$ be an $lsa$-space, and let $U\subseteq M$ be open. Let $\mathfrak{S}(U)=\{V\in \mathfrak{S}(M)\mid V\subseteq U\}$ and let $\mathcal O_U$ be the restriction of $\mathcal O_M$ onto $U$. Then the triple $(U,\mathfrak{S}(U),\mathcal O_U)$ is an $lsa$-space, too, and the inclusion map $U\to M$ is a morphism in $\mathcal{LSA}_K$. Therefore for every map $f:X\to Y$ of $lsa$-spaces and for every open $U\subseteq Y$ the map $f:f^{-1}(U)\to U$ is a  morphism in
$\mathcal{LSA}_K$. A map $f:X\to Y$ of $lsa$-spaces is {\it locally trivial} if there is admissible cover $\{U_i\mid i\in I\}$ of $Y$ and for every $i\in I$ an $lsa$-space $X_i$ and an isomorphism $\phi_i:U_i\times X_i\to f^{-1}(U_i)$ such that $f\circ\phi_i:U_i\times X_i\to U_i$ is the projection onto the first coordinate. 
\end{defn}
\begin{thm}\label{fibration_theorem1} Let $f:X\to Y$ be a locally trivial map of $lsa$-spaces. Then the map $S_*(f):S_*(X)\to S_*(Y)$ induced by $f$ is a fibration. 
\end{thm}
\begin{proof} Let $\{U_i\mid i\in I\}$ be an admissible cover of $Y$ such that for every $i$ there is an $lsa$-space $X_i$ and an isomorphism $\phi_i:U_i\times X_i\to f^{-1}(U_i)$ such that $f\circ\phi:U_i\times X_i\to f^{-1}(U_i)$ is the projection onto the second coordinate. For every $i\in I$ the inclusion map $U_i\to Y$ induces an inclusion $S_*(U_i)\to S_*(Y)$ of simplicial sets. Let $\mathcal U$ be the collection of these simplicial subsets of $S_*(Y)$.
\begin{lemma}\label{3.13} The class of cofibrations $i:A_*\to B_*$ which has the
$\mathcal U$-subordinate left lifting property with respect to $S_*(f):S_*(X)\to S_*(Y)$ contains all trivial cofibrations. 
\end{lemma}
\begin{proof} By Corollary \ref{3.3} we have to verify that for every $i\in I$ the simplicial map:
$$f_*|_{f_*^{-1}(S_*(U_i))}:f_*^{-1}(S_*(U_i))\to S_*(U_i)$$
is fibrant. Since $f_*^{-1}(S_*(U_i))=S_*(f^{-1}(U_i))$, there is a commutative diagram:
$$\xymatrix{ S_*(U_i)\times S_*(X_i) \ar[r]^{\cong}\ar[rd]^{\pi} &S_*(U_i\times X_i)\ar[r]^{S_*(\phi_i)}\ar[d] & S_*(f^{-1}(U_i))
\ar[ld]_{S_*(f)} \\
& S_*(U_i) & }$$
where the left horizontal map exists because the functor of semi-algebraic singular simplexes commutes with direct products, the map $\pi$ is the projection onto the first factor, and the middle vertical arrow is induced by the projection $U_i\times X_i\to U_i$. Since $\phi_i$ is an isomorphism, the map $S_*(\phi_i)$ is also an isomorphism, so it will be enough to show that $\pi$ is fibrant. But this is true, since $S_*(X_i)$ is fibrant by Lemma \ref{2.6}.
\end{proof}
\begin{defn} By a pair of closed simplicial complexes $(A,B)$ we mean two closed simplicial complexes $A=(A_v,A_s)$ and $B=(B_v,B_s)$ such that $B$ is a subcomplex of $A$, that is, we have $B_v\subseteq A_v$ and $B_s\subseteq A_s$. They form a category $\mathcal{SC}^p$ where a morphism of pairs of closed simplicial complexes $(A,B)\to(A',B')$ is a map of closed simplicial sets $\phi:A_v\to A'_v$ such that the restriction of $\phi$ onto $B_v$ is a map $B_v\to B_v'$ of closed simplicial sets. Note that when $(A,B)$ is a pair of closed simplicial complexes then
$|B|_K$ is closed semi-algebraic subset of $|A|_K$, and consequently the rule $(A,B)\mapsto(|A|_K,|B|_K)$ furnishes a functor $\mathcal{SC}^p\to\mathcal{LSA}^p_K$.
\end{defn}
\begin{lemma}[Woerheide's trick]\label{woerheide} Let $\{V_j\mid j\in J\}$ be an admissible cover of $\Delta^n_K$. Then there is a pair of finite closed simplicial complexes $(A,B)$ and an isomorphism $\tau:(|A|_K,|B|_K)\to (\Delta^n_K,\Lambda^n_{k,K})$ in
$\mathcal{LSA}^p_K$ such that the realisation of every simplex of $A$ lies in a set $\tau^{-1}(V_j)$ for some $j\in J$.
\end{lemma}
\begin{proof} This is explained in the second paragraph of the sixth section of Woerheide's thesis \cite{Wo}, on page 36. \end{proof}
Let us start the final part of the proof of Theorem \ref{fibration_theorem1} now! By definition we have to check that for every diagram:
$$\xymatrix{ \Lambda^n_{k,K}\ar[r]^{f}\ar[d]_{i^n_k} & X\ar[d]^{f} \\
\Delta^n_K\ar[r]^{g} \ar@{.>}[ru]^{h}& Y}$$
in $\mathcal{LSA}_K$ there is a morphism $h:\Delta^n_K\to X$ of $lsa$-spaces which makes the diagram commutative. By Lemma \ref{woerheide} there is a  pair of finite closed simplicial complexes $(A,B)$ and an isomorphism $\tau:(|A|_K,|B|_K)\to (\Delta^n_K,\Lambda^n_{k,K})$ in $\mathcal{LSA}^p_K$ such that the realisation of every simplex of $A$ lies in a set $(g\circ\iota)^{-1}(U_i)$ for some $i\in I$. By Remark \ref{2.13a} the simplicial sets $S_*(\Delta^n_K)$ and $S_*(\Lambda^n_{k,K})$ are contractible, therefore the map $S_*(\Delta^n_K)\to S_*(\Lambda^n_{k,K})$ induced by the inclusion $\Lambda^n_{k,K}\subset\Delta^n_K$ is a weak equivalence. Since the pair $(\Delta^n_K,\Lambda^n_{k,K})$ is isomorphic to the pair $(|A|_K,|B|_K)$ via
$\tau$, the map $S_*(|B|_K)\to S_*(|A|_K)$ induced by the inclusion $|B|_K\subset|A|_K$ is also a weak equivalence. Using Proposition \ref{3.8} we get that the inclusion $B\subset A$ is a weak equivalence, too. Therefore the map $\epsilon:B_*\to A_*$ induced by the inclusion $B\subset A$ is a trivial cofibration. We have a commutative diagram:
$$\xymatrix{ B_*\ar[rr]^{S_*(f\circ\tau)\circ\iota_B}\ar[d]_{\epsilon} & &S_*(X)\ar[d]^{S_*(f)} \\
A_*\ar[rr]^{S_*(g\circ\tau)\circ\iota_A} \ar@{.>}[rru]^{\widetilde h}& & S_*(Y).}$$
By Proposition \ref{3.10} it will be enough to show that there is a map $\widetilde h:A_*\to S_*(X)$ of simplicial sets which makes the diagram commutative. However $g:A_*\to S_*(Y)$ is subordinate to $\mathcal U$ so the latter follows from Lemma \ref{3.13}.
\end{proof}

\section{Coverings and the fundamental group}

The following are all taken from \cite{DeKn1}. 
\begin{defn} A locally semialgebraic map $f:M\to N$ over $K$ is called a local isomorphism if every point $x\in M$ has an open locally semi-algebraic neighborhood $U$ such that the restriction $f|_U:U\to f(U)$ is a locally semialgebraic isomorphism from $U$ onto an open locally semialgebraic subset of $N$.
\end{defn}
\begin{thm} Let $f:M\to N$ be a local isomorphism between locally semialgebraic spaces $M, N$ over $K$. Then $M$ has an admissible covering $\{M_{\alpha}\mid\alpha\in I\}$ such that $M_{\alpha}$ is mapped isomorphically onto an open locally semialgebraic subset of $N$ by $f$ for every
$\alpha\in I$.
\end{thm}
\begin{defn} A locally semialgebraic map $p:M\to N$ is called a locally semialgebraic covering of $N$ (or "covering" in short), if $p$ is surjective, locally trivial, and has discrete fibres. The map $p$ is called connected if $M$, and hence also $N$, is connected.
\end{defn}
\begin{thm} Let $p:M\to N$ be a connected locally semialgebraic covering, and let $x_0\in M,\ y_0=p(x_0)$.
\begin{enumerate}
\item[$(i)$] The induced homomorphism $p_*:\pi_1(M,x_0)\to
\pi_1(N,y_0)$ is infective. If $p$ is finite, the index of
$\pi_1(M,x_0)$ in $\pi_1(N,y_0)$ is equal to the degree of
$p$.
\item[$(ii)$] If $f:(L,z_0)\to(N,y_0)$ is a locally semialgebraic map with $L$ connected then $f$ can be lifted to a locally semialgebraic map $g:(L,z_0)\to(M,x_0)$ with $p\circ g=f$ if and only if $f_*(\pi_1(L,z_0))\subseteq p_*(\pi_1(M,x_0))$. The map $g$ is uniquely determined provided it exists. If $f$ is a covering, then $g$ is also a covering.
\end{enumerate}
\end{thm}
In particular the theorem above says that every connected covering $p:M\to N$ of $N$ is determined up to isomorphism by the subgroup $p_*(\pi_1(M,x_0))$ of $\pi_1(N,y_0)$. Conversely we can prove the next theorem.
\begin{thm} Let $N$ be a connected locally semialgebraic space, $y_0\in N$ and $H$ be a subgroup of $\pi_1(N,y_0)$. Then there exists a connected locally semialgebraic covering $p:M\to N$ with $p_*(\pi_1(M,x_0))=H$ for some $x_0\in M$ with $p(x_0)=y_0$.
\end{thm}
Applying this theorem to the trivial subgroup of $\pi_1(N,y_0)$ we derive the following corollary.
\begin{cor} Let N be a connected locally semialgebraic space and $y_0\in N$.
\begin{enumerate}
\item[$(i)$] Up to isomorphism there exists a uniquely determined connected locally semialgebraic covering $p:M\to N$, called the universal covering of $N$ with the following universal property. For every $x_0\in M$ with $p(x_0)=y_0$ and every connected covering $f:(M',z_0)\to (N,y_0)$ there exists a uniquely determined locally semialgebraic map
$g:(M,x_0)\to(M',z_0)$ such that the diagram
$$\xymatrix{ M\ar@{.>}[r]^{g}\ar[d]_{p}& M'\ar[dl]^f \\
N & }$$
communes. The map $g$ is again a covering.
\item[$(ii)$] A connected covering $p':M'\to N$ is universal and hence isomorphic to $p:M\to N$, if and only if $M'$ is simply connected.
\end{enumerate}
\end{cor}

\section{The fibration theorem: geometric fibrations}\label{fibration_section2}

\begin{defn}\label{4.1} In this paper we will freely use the basic notions and results from the theory of Nash manifolds. Informally speaking the concept of Nash manifolds is the semi-algebraic analogue of differentiable manifolds in classical topology and it has a similar, highly developed theory. Our main reference is the book \cite{BCR}. Now we review some more specified definitions and results in this theory which will play an important role in this section. Let $K$ be a real closed field, as above. A Nash map $f:X\to Y$ of Nash manifolds over $K$ is {\it Nash trivial} if there is a Nash manifold $Z$ and a Nash diffeomorphism $\phi_i:Z\times Y\to X$ (over $K$) such that $f\circ\phi_i:U_i\times X_i\to U_i$ is the projection onto the first coordinate. 
\end{defn}
\begin{thm}\label{thom_isotropy1} Let $f:X\to K^n$ be a proper surjective submersion of Nash manifolds over $K$. Then $f$ is Nash trivial. 
\end{thm}
\begin{proof} See Proposition 5.2 of \cite{CS} on page 368 and the main theorem of \cite{Es} on page 1209.
\end{proof}
\begin{lemma}\label{nash_covers} Let $N$ be a Nash manifold over $K$. Then we can find a finite covering of $N$ by open semi-algebraic subsets which are Nash diffeomorphic to some affine space $K^n$.
\end{lemma}
\begin{proof} This is Lemma 3.2 of \cite{Es} on page 1217.
\end{proof}
\begin{defn} A map $f:X\to Y$ of Nash manifolds is {\it locally Nash trivial} if there is admissible cover $\{U_i\mid i\in I\}$ of $Y$ and for every $i\in I$ a Nash manifold $X_i$ and a Nash diffeomorphism
$\phi_i:U_i\times X_i\to f^{-1}(U_i)$ such that $f\circ\phi_i:U_i\times X_i\to U_i$ is the projection onto the first coordinate. 
\end{defn}
\begin{cor}\label{thom_isotropy2} Let $f:X\to B$ be a proper surjective submersion of Nash manifolds. Then $f$ is locally Nash trivial. 
\end{cor}
\begin{proof} Immediate from Theorem \ref{thom_isotropy1} and Lemma \ref{nash_covers}.
\end{proof}
\begin{defn}\label{4.6} Let $f:X\to B$ be a surjective submersion of Nash manifolds over $K$.  We say that $f$ is finitely compactifiable if it can be embedded into a commutative diagram
$$\xymatrix{ X\ar[r]^{j}\ar[rd]_{f} & Y\ar[d]^{h} & Z\ar[l]_{i}\ar[ld]^{g}\\
 & B & }$$
in which:
\begin{enumerate}
\item[$(a)$] the map $j$ is an open immersion, the set $j(X)$ is dense in every fibre of $h$, and $Y=i(Z)
\cup j(X)$;
\item[$(b)$] the map $h$ is a proper surjective submersion of Nash manifolds, 
\item[$(c)$] the map $g$ is proper and locally a Nash isomorphism.
\end{enumerate}
\end{defn}
In order to handle elementary fibrations, we will need the following mild generalisation:
\begin{thm}\label{thom_isotropy3} Let $f:X\to B$ be a finitely compactifiable surjective submersion of Nash manifolds over  $K$. Then $f$ is locally Nash trivial.\qed
\end{thm}
In the proof of the theorem above by Lemma \ref{nash_covers} we may assume that $B$ is Nash isomorphic to an affine space $K^n$. Let 
$$\xymatrix{ X\ar[r]^{j}\ar[rd]_{f} & Y\ar[d]^{h} & Z\ar[l]_{i}\ar[ld]^{g}\\
 & B & }$$
be the type of diagram which is in Definition \ref{4.6}. By Lemma \ref{nash_covers} we know that both $h$ and $g$ are Nash trivial. Unfortunately the proof of the claim in this case is still quite technical, so we will relegate this argument to the second part of this series. 
\begin{defn}\label{4.8} For every field $F$ let $\mathcal{AF}_F$ denote the category of affine algebraic varieties over $F$. Let $M\subseteq\mathbb A_K^n$ be an affine variety over $K$. Then $M(K)$ is also a semi-algebraic set, and therefore we may equip $M(K)$ with the structure of a semi-algebraic space. This structure is independent of the choice of the embedding into affine space, and furnishes a functor $\mathcal{AF}_K\to\mathcal{SA}_K$ which we will simply denote by $M\mapsto M(K)$ by slight abuse of notation. 
\end{defn}
\begin{defn}\label{4.9} For every field $F$ let $\mathcal{VAR}_F$ denote the category of algebraic varieties over $F$. Let $M$ be an algebraic variety over $K$. Then there is a finite covering $M=\bigcup_{i\in I}M_i$ by Zariski-open affine subvarieties such that the intersections $M_i\cap M_j$ are also affine. By the above each $M_i(K)$ is canonically equipped with the the structure of a semi-algebraic space and the induced ringed space structure on the open subset $M_i(K)\cap M_j(K)\subseteq M_i(K)$ is the canonical semi-algebraic space structure on the affine algebraic set $M_i(K)\cap M_j(K)$. Therefore we may patch the ringed spaces $M_i(K)$ together to equip $M(K)$ with the structure of a locally semi-algebraic space. (See Lemma 2.1 of \cite{DeKn2} on page 12 for the details of this construction.) This structure is independent of the choice of the covering chosen, and furnishes a functor $\mathcal{VAR}_K\to\mathcal{LSA}_K$ which we will simply denote by $M\mapsto M(K)$ by slight abuse of notation. Note that this construction extends the functor $\mathcal{AF}_K\to\mathcal{SA}_K$ introduced in Definition \ref{4.8} above. 
\end{defn}
\begin{lemma}\label{para} Let $M$ be an object of
$\mathcal{VAR}_K$. Then the $lsa$-space $M(K)$ is paracompact. If $M$ is quasi-projective, then $M(K)$ is regular, too.  
\end{lemma}
\begin{proof} Since $M(K)$ is has a finite admissible cover by open semi-algebraic subspaces, it is paracompact. Since the projective space $\mathbb P^n(K)$ has a closed semi-algebraic embedding into $K^{n^2}$ (see Theorem 3.4.4 of \cite{BCR} on page 72), we get that $M(K)$ has a locally closed semi-algebraic embedding into $K^{n^2}$, if $M$ is quasi-projective. Since every locally closed semi-algebraic subspace of $K^n$ is regular, the second claim follows.
\end{proof}
\begin{rems}\label{4.11} $(i)$ Let $L$ be a real closed extension of $K$, and let $M$ be a quasi-projective variety over $K$. Then we have a natural isomorphism
$$M(L)\cong e_{L/K}(M(K)),$$
and therefore a natural weak equivalence $S_*(M(K))\to S_*(M(L))$ by Lemma \ref{para} and part $(i)$ of Theorem \ref{basic_comparison}. Similarly for every $m\in M(K)$ there is a natural weak equivalence $S_*(M(K),s_m)\to S_*(M(L),s_m)$ in $SSets^0$.

$(ii)$ Assume now that $K$ is the real number field $\mathbb R$. Then there is a natural weak equivalence $S_*(M(\mathbb R))\to S^{top}_*(M(\mathbb R))$ in $SSets$, and for every
$m\in M(\mathbb R)$ a natural weak equivalence $S_*(M(\mathbb R),s_m)\to S^{top}_*(M(\mathbb R),s_m)$ in $SSets^0$ by Lemma \ref{para} and part $(ii)$ of Theorem \ref{basic_comparison}.
\end{rems}
\begin{defn} Let $L/F$ be a finite extension of fields, and let $X$ be a variety defined over $L$. The contravariant functor
$\mathrm{Res}_{L/F}X$ from the category of $F$-schemes to sets is defined by
$$\mathrm{Res}_{L/F}X(S) = X(S \times_FL).$$
The variety that represents this functor is called the restriction of scalars, and is unique up to unique isomorphism if it exists. 
\end{defn}
\begin{prop}\label{blr} Let $L/F$ be a finite extension of fields, and let $X$ be an $L$-variety. If every finite subset of $X$ is contained in some affine open subset of $X$ then
$\mathrm{Res}_{L/F}X$ is representable.
\end{prop}
\begin{proof} This is a special case of Theorem 4 of \cite{BLR} in \S7.6. 
\end{proof}
\begin{defn}\label{4.14} Now let $X$ be an algebraic variety over $C$ such that $\mathrm{Res}_{C/K}X$ is representable. By Proposition \ref{blr} above this is the case for example when $X$ is quasi-projective. Note that $X(C)=
\mathrm{Res}_{C/K}X(K)$ so we get that $X(C)$ has the structure of an $lsa$-space which is both regular and paracompact. When $X$ is any variety we can still equip $X(C)$ with the structure of an $lsa$-space by covering $X$ with finitely many Zariski-open affine subvarieties whose intersections are also affine, and use the same construction as in Definition \ref{4.9}. We will denote this $lsa$-space by$X_{lsa}$, or when there could be no confusion, simply by $X$. We get a functor $\mathcal{VAR}_C\to\mathcal{LSA}_K$. For every $X$ as above we will call $S_*(X)$ the {\it topological type} of $X$ and $\pi_0(X)$ is the {\it set of connected components} of $X$ (or of $X(C)$). Now let $m\in X(C)$. Then $(X,m)$ is a pointed $lsa$-space. For every positive integer $n$ we will call $\pi_n(X,m)$ the {\it $n$-th homotopy group} of the pair $(X,m)$. When we want to distinguish these groups, etc.~from other similar constructions, we will add the adjective semi-algebraic. 
\end{defn}
For every real closed field $L$ let $L(\mathbf i)=L(\sqrt{-1})$ be its algebraic closure. Then for every extension $L/K$ of real closed fields we have a corresponding extension $L(\mathbf i)/
K(\mathbf i)$ of algebraically closed fields.
\begin{thm}\label{basic_comparison2} $(i)$ Let $L$ be a real closed extension of $K$, and let $X$ be a quasi-projective variety over $K(\mathbf i)$. There is a natural weak equivalence $S_*(X)\!\to\! S_*(X_{L(\mathbf i)})$, and for every $m\in X(K(\mathbf i))$ there is a natural weak equivalence $S_*(X,s_m)\to S_*(X_{L(\mathbf i)},s_m)$ in $SSets^0$.

$(ii)$ Assume now that $K$ is the real number field $\mathbb R$. Then there is a natural weak equivalence $S_*(X)\to S^{top}_*(X)$ in $SSets$, and for every
$m\in X(\mathbb C)$ a natural weak equivalence $S_*(X,s_m)\to S^{top}_*(X,s_m)$ in $SSets^0$.
\end{thm}
\begin{proof} Immediate from Remark \ref{4.11}.
\end{proof}
\begin{lemma}\label{conn1} Let $X$ be a connected quasi-projective rational curve over $C$. Then the the semi-algebraic set $X(C)$ is connected. 
\end{lemma}
\begin{proof} Note that it will be sufficient to prove that $X(C)-x$ is connected for every $x\in X(C)$ since the union of these is $X(C)$, and any two has a non-empty intersection. Therefore we may assume without the loss of generality that $X=\mathbb A^1_C-S$, where $S\subset\mathbb A^1(C)$ is a finite set. Then $X(C)$ is isomorphic to $K^2-R$ for some finite set $R\subset K^2$ as a semi-algebraic space. Clearly it will be enough to show that the latter is connected. Let $x,y\in K^2-R$ be two arbitrary points. Since $K$ is infinite, there are $K$-linear lines $L,M\subset K^2$ containing $x$ and $y$, respectively, such that $L$ and $M$ do not intersect $R$ and they are not parallel. Let $z$ be the intersection of $L$ and
$M$, and let $T$ be the union of the closed segment of $L$ between $x$ and $z$, and the closed segment of $M$ between $z$ and $y$. Then $T$ is a connected semi-algebraic set in $K^2-R$, and therefore $x$ and $y$ are in the same connected component of $K^2-R$.
\end{proof}
\begin{lemma}\label{conn2} Let $X$ be a connected quasi-projective algebraic curve over $C$. Then the semi-algebraic set $X(C)$ is connected. 
\end{lemma}
\begin{proof} Let $g$ be the genus of the unique smooth, projective, irreducible curve $\widetilde X$ containing $X$ as an open subscheme. Let $S\subseteq X(C)$ be a semi-algebraic connected component. Since $S$ is open in the strong topology, by the semi-algebraic version of the implicit function theorem (see Corollary 2.9.8 of \cite{BCR} on pages 56-57) the set $S$ is infinite. In particular there are $2g+2$ different points $p_1,p_2,\ldots$ in $S$. Let $D$ be the divisor $p_1+p_2+\cdots+p_{2g+2}$ and consider the line bundle $\mathcal O(D)$ on $\widetilde X$. By a routine application of the Riemann--Roch theorem we get that the full linear system of $\mathcal O(D)$ furnishes a projective embedding of $\widetilde X$ into $\mathbb P^{g+3}_C$. Now let $x\in X(C)$ be arbitrary. By Bertini's hyperplane section theorem there is a hyperplane $H\subset\mathbb P^{g+3}_C$ which intersects $\widetilde X$ transversally, contains $x$, and does not contain any point of the complement $\widetilde X-X$. The pencil spanned by $D$ and $H\cap\widetilde X$ furnishes a map $f:\widetilde X\to\mathbb P^1_C$ such that the pre-image of $0$ is $D$, the pre-image of $\infty$ is $H\cap\widetilde X$, and $f$ is \'etale at each point of $D$ and $H\cap\widetilde X$. There is a non-empty open subcurve $U\subseteq\mathbb P^1_C$ which contains $0,f(x)$, and the restriction of $f$ onto $f^{-1}(U)$ is a finite, \'etale map onto $U$. By Lemma \ref{conn1} there is a continuous semi-algebraic map $p:I_K\to U(C)$ such that $p(0)=0$ and $p(1)=\infty$. The pull-back of the map $f^{-1}(U)(C)\to U(C)$ with respect to $p$ a covering of $I_K$. Since $I_K$ is contractible, this covering is trivial, so it has a semi-algebraic section. We get that here is a continuous semi-algebraic map $\widetilde p:I_K\to f^{-1}(U)(C)$ such that $p(1)=x$ and $p(1)\in f^{-1}(0)\subset S$. Therefore $x$ also lies in $S$, and hence $X(C)$ is connected. 
\end{proof}
\begin{prop}\label{connectivity} Let $X$ be a connected quasi-projective algebraic variety over $C$. Then the semi-algebraic set $X(C)$ is connected. 
\end{prop}
\begin{proof} It will be enough to show that for any $x,y\in X(C)$ there is a connected semi-algebraic subset $S\subseteq X(C)$ containing both. By the resolution of singularities there is a surjective map $f:\widetilde X\to X$ such that $\widetilde X$ is quasi-projective, smooth and connected. Since $f$ is surjective, there are $\widetilde x,\widetilde y\in\widetilde X(C)$ such that $f(\widetilde x)=x$ and $f(\widetilde y)=y$. Since the image of a connected semi-algebraic set under a semi-algebraic map is both connected and semi-algebraic, it will be enough to show that there is a connected semi-algebraic subset $S\subseteq\widetilde X(C)$ containing both
$\widetilde x$ and $\widetilde y$. In other words we may assume that $X$ is smooth. By Bertini's hyperplane section theorem there is a connected smooth quasi-projective curve containing both $x$ and $y$. The proposition now follows from Lemma \ref{conn2}.
\end{proof}
\begin{defn} A finitely compactifiable nice map over a field $F$ is a morphism of smooth $F$-varieties $f:U\to S$ that can be embedded into a commutative diagram
$$\xymatrix{ U\ar[r]^{j}\ar[rd]_{f} & Y\ar[d]^{h} & Z\ar[l]_{i}\ar[ld]^{g}\\
 & S & }$$
in which:
\begin{enumerate}
\item[$(a)$] the map $j$ is an open immersion, the set $j(U)$ is dense in every fibre of $h$, and $Y=i(Z)
\cup j(U)$;
\item[$(b)$] the map $h$ is smooth and projective, with geometrically irreducible fibres,
\item[$(c)$] the map $g$ is finite and \'etale, and each fibre of $g$ is nonempty.
\end{enumerate}
\end{defn}
\begin{thm}\label{fibration-theorem2} Let $f:X\to B$ be finitely compactifiable nice map of smooth varieties over $C$. Then the map $f_*:S_*(X)\to S_*(B)$ induced by $f$ is fibration.
\end{thm}
\begin{proof} Let $M$ be a smooth algebraic variety over $C$. Then we can equip $M(C)$ with the structure of a Nash manifold over $K$ as follows. There is a finite covering $M=\bigcup_{i\in I}M_i$ by Zariski-open affine subvarieties such that the intersections $M_i\cap M_j$ are also affine. Choose an embedding $\phi_i:M_i\to\mathbb A^{m_i}_C$ for each $i$. Then $\phi_i(M_i)(C)\subseteq C^{m_i}\cong K^{2m_i}$ has the structure of a Nash manifold, as $M_i$ is smooth. Moreover for every $i,j\in I$ the image $\phi_i(M_i\cap M_j)(C)$ is open, and the transition map $\phi_i(M_i\cap M_j)(C)\to\phi_j(M_i\cap M_j)(C)$ is a Nash-diffeomorphism, so we get a Nash atlas on $M(C)$. This structure is independent of the choice of the covering and the embeddings chosen. We get a functor from $\mathcal{VAR}_C$ to the category of Nash manifolds over $K$ such that its composition with the forgetful functor from the category of Nash manifolds over $K$ to $\mathcal{LSA}_K$ is the functor $\mathcal{VAR}_K\to\mathcal{LSA}_K$ in Definition \ref{4.14} above. By our assumptions and the Nullstellensatz $f:X(C)\to B(C)$ is a finitely compactifiable surjective submersion of Nash manifolds over  $K$. We get from Theorem \ref{thom_isotropy3} that this map $f:X(C)\to B(C)$ is locally Nash trivial, and hence the map of underlying $lsa$-spaces $X_{lsa}\to B_{lsa}$ is locally trivial. The claim now follows from Theorem \ref{fibration_theorem1}.
\end{proof}
Since $S_*(X)$ and $S_*(B)$ are fibrant, and the fibre of $f_*$ over any $s_x\in  S_0(X)$ is $S_*(f^{-1}(x))$ we get the following 
\begin{cor}\label{fibres_are_homotopic} Let $f:X\to B$ be as above and assume that $B$ and the fibres of $f$ are connected and quasi-projective. Then for every $x\in X(C)$ the homotopy type of $f^{-1}(x)$ is the homotopy fibre of the map $f_*:S_*(X)\to S_*(B)$. In particular for every $x,y\in X(C)$ the simplicial sets $S_*(f^{-1}(x))$ and $S_*(f^{-1}(y))$ are weakly equivalent.
\end{cor}
\begin{proof} Follows at once from Theorem \ref{fibration-theorem2} and Proposition \ref{connectivity}. 
\end{proof}
Another important consequence is the following
\begin{cor}\label{homotpy_exact_sequence} Let $f:X\to B$ be as above and assume that $B$ and the fibres of $f$ are connected and quasi-projective. Then there is a homotopical long exact sequence:
$$\cdots\to\pi_2(X,x)\to\pi_2(B,f(x))\to\pi_1(f^{-1}(x),x)\to\pi_1(X,x)\to\pi_1(B,f(x))\to1$$
for every $x\in X(C)$.\qed
\end{cor}

\section{The homotopy type of curves and abelian varieties}

\begin{defn} Let $(k,<)$ be an ordered field. A Dedekind cut in $k$ is a subset $A\subseteq k$ such that if $x\in A$ and $y<x$, then $y\in A$. When $k$ is a real closed field then the set of Dedekind cuts of $k$ is bijective with the set of orderings on $k(x)$. The correspondence assigns to every ordering $<$ on $k(x)$ the Dedekind cut $A_<=\{y\in k\mid y< x\}$ (see Exercise 24 of Section 2 of \cite{Bour}.)
\end{defn}
\begin{lemma}\label{completion} Let $k$ be a real closed field, and let $K/k,L/k$ be two real closed extensions of $k$. Then there is a real closed extension $M/k$ which contains sub-extensions isomorphic to $K/k$ and $L/k$, respectively. 
\end{lemma}
\begin{proof} By transfinite induction there is a cardinal $\kappa$ and a chain $\{K_{\alpha}\}_{\alpha\in\kappa}$ of subfields of $K$ indexed by $\kappa$ such that
\begin{enumerate}
\item[$(i)$] we have $K_0=k$,
\item[$(ii)$] for every limit ordinal $\alpha\in\kappa$ we have $K_{\alpha}=\bigcup_{\beta\in\alpha}
K_{\beta}$,
\item[$(iii)$] for every $\alpha\in\kappa$ we have $K_{\alpha+1}=K_{\alpha}(x_{\alpha})$ for some
 $x_{\alpha}\in K_{\alpha}$.
\end{enumerate}
Let $<,\prec$ denote the unique ordering on the real closed fields $K$ and $L$, respectively. For every $\alpha\in\kappa$ let $<_{\alpha}$ denote the restriction of $<$ onto $K_{\alpha}$. By transfinite induction we are going to define a chain $\{L_{\alpha}\}_{\alpha\in\kappa}$ of real closed extensions of $L$ indexed by $\kappa$ and a set $\{i_{\alpha}\}_{\alpha\in\kappa}$ of $k$-embeddings $i_{\alpha}:K_{\alpha}\to L_{\alpha}$ such that
\begin{enumerate}
\item[$(i)$] for every $\beta\in\alpha\in\kappa$ we have $i_{\alpha}|_{K_{\beta}}=i_{\beta}$,
\item[$(ii)$] for every $\alpha\in\kappa$ the map $i_{\alpha}$ is order-preserving with respect to $<_{\alpha}$ and the unique ordering $\prec_{\alpha}$ on $L_{\alpha}$.
\end{enumerate}
This is sufficient to conclude the claim, since the union $M=\bigcup_{\alpha\in\kappa}L_{\alpha}$ will have the required properties. Let $L_0=L$ and let $i_0$ be the inclusion map $k\to L$. For every limit ordinal $\alpha\in\kappa$ we set $L_{\alpha}=\bigcup_{\beta\in\alpha}L_{\beta}$ and $i_{\alpha}=\bigcup_{\beta\in\alpha}i_{\beta}$; since the ordering on $k$ is unique, this map is order-preserving. Since the union of a chain of real closed fields is real closed and the union of order-preserving maps is order-preserving, the pair $(L_{\alpha},i_{\alpha})$ has the required properties.

Next we are going to consider the case of those $\alpha\in\kappa$ such that $\alpha=\beta+1$. Let $K_{\beta}^{\vee},K_{\beta}^{\wedge}$ be the algebraic closure of $K_{\beta}$ and $i_{\beta}(K_{\beta})$ in $K$ and $L_{\beta}$, respectively. Since $K$ and $L_{\beta}$ are real closed, the fields $K_{\beta}^{\vee}$ and $K_{\beta}^{\wedge}$ are the real closures of
$K_{\beta}$ and $i_{\beta}(K_{\beta})$ with respect to the ordering $<_{\beta}$ and the restriction of $\prec_{\beta}$ onto $i_{\beta}(K_{\beta})$, respectively. Because $i_{\beta}$ is order-preserving there is an extension $j_{\beta}:K_{\beta}^{\vee}\to K_{\beta}^{\wedge}$ of $i_{\alpha}$, which is an isomorphism, using the uniqueness of the real closure. 

By assumption there is an $x\in K_{\alpha}$ such that $K_{\alpha}=K_{\beta}(x)$. If $x\in K_{\beta}^{\vee}$ then set $L_{\alpha}=L_{\beta}$ and $i_{\alpha}=j_{\beta}|_{K_{\beta}(x)}$. Clearly $(L_{\alpha},i_{\alpha})$ has the required properties. Otherwise $x$ is transcendental over $K_{\beta}^{\vee}$. Let $<'_{\alpha}$ denote the restriction of $<$ onto $K_{\beta}^{\vee}(x)$ and let $A=\{y\in K_{\beta}^{\vee}\mid y<'_{\alpha}x\}$ be the corresponding Dedekind cut. Consider the set
$$B=\{y\in L_{\beta}\mid 
\exists z\in K_{\beta}^{\vee}\textrm{ such that }y\preceq_{\beta}j_{\beta}(z)\}.$$
Note that $B$ is a Dedekind cut in $L_{\beta}$ and $B\cap K_{\beta}^{\wedge}=
j_{\beta}(K_{\beta}^{\vee})$. Let $\prec'_{\beta}$ be the ordering on the rational function field $L_{\beta}(x')$ corresponding to $B$. By the above the unique embedding $j'_{\beta}:K_{\beta}^{\vee}(x)\to L_{\beta}(x')$ of $i_{\alpha}$ such that $j_{\beta}'(x)=x'$ and $j'_{\beta}|_{K_{\beta}^{\vee}}=j_{\beta}$ is order-preserving with respect to the orderings  $<_{\alpha}$ and $\prec'_{\beta}$. Let $L_{\alpha}$ be the real closure of $L_{\beta}(x')$ with respect to $\prec'_{\beta}$ and let $i_{\alpha}:K_{\alpha}
\to L_{\alpha}$ be the composition of $j'_{\beta}|_{K_{\alpha}}$ and the inclusion $L_{\beta}(x')\subset L_{\alpha}$. Again it is clear that $(L_{\alpha},i_{\alpha})$ satisfies the required properties. 
\end{proof}
\begin{defn} Recall that a smooth, geometrically connected curve $X$ defined over a field has type $(g,d)$ if $g$ is the genus of the smooth projective completion $X^c$ of $X$ and $d$ is the number of geometric points in the complement of $X$ in $X^c$. We define a surface group $\Gamma_{g,d}$ to be the group given the by presentation:
$$\Gamma_{g,0}=\langle a_1,b_1,a_2,b_2,\ldots,a_g,b_g\mid [a_1,b_1]\cdots [a_g,b_g] = 1
\rangle$$
when $d=0$, and the free group on $2g+d-1$ generators, otherwise.
\end{defn}
\begin{prop}\label{curves_fungroup} Let $X$ be a smooth, geometrically connected curve over $C$ of type $(g,d)$ and let $x\in X(C)$. Then the fundamental group $\pi_1(X,x)$ is isomorphic to $\Gamma_{g,d}$.  
\end{prop}
\begin{proof} Note that every real closed field is the extension of the unique real closure $\mathbb Q^{real}$ of $\mathbb Q$. Therefore by Lemma \ref{completion} above there is a real closed extension $L$ of $K$ which contains a copy of $\mathbb R$. Since $\pi_1(X_{L(\mathbf i)},x)\cong\pi_1(X,x)$ by part $(i)$ of Theorem \ref{basic_comparison2}, we may assume without the loss of generality that $K$ is an extension of $\mathbb R$. Because of the existence of a smooth connected moduli stack for smooth geometrically irreducible projective curves of genus $g$ pointed with $d$ different points there is an elementary fibration $f:T\to M$ of smooth varieties over $C$ such that $M$ is connected, the geometric fibres of $f$ are curves of type
$(g,d)$ and every such curve over $C$ is isomorphic to one of these fibres.

Let $Y$ be a smooth, geometrically connected curve of type $(g,d)$ defined over the subfield $\mathbb C\subseteq C$ and let $y\in Y(\mathbb C)$. By part $(ii)$ of Theorem \ref{basic_comparison2} we know that $\pi_1(Y,y)\cong\pi_1(S^{top}_*(Y),s_y)$ and since $\pi_1(S^{top}_*(Y),s_y)\cong\Gamma_{g,d}$, we get that the claim holds for $Y$. On the other hand by Corollary \ref{fibres_are_homotopic} we have $\pi_1(X,x)\cong\pi_1(Y,y)$, so the claim holds for $X$, too.
\end{proof}
\begin{prop}\label{curves_k1} Let $X$ be a smooth, geometrically connected curve over $C$ which is not a projective curve of genus zero, and let $x\in X(C)$. Then $S_*(X,s_x)$ is weakly  equivalent to $B\pi_1(X,x)$.
\end{prop}
\begin{proof} We may argue as above and conclude that we may assume without the loss of generality that $K$ contains $\mathbb R$. Let $Y$ be again a smooth, geometrically connected curve of type $(g,d)$ defined over the subfield $\mathbb C\subseteq C$ and let $y\in Y(\mathbb C)$. By part $(ii)$ of Theorem \ref{basic_comparison} we know that
$\pi_n(Y,y)\cong\pi_n(S^{top}_*(Y),s_y)$ for every $n\geq2$, and since $\pi_n(S^{top}_*(Y),s_y)\cong 0$ for every such $n$, we get that the claim holds for $Y$. On the other hand we may repeat the proof above to show that $\pi_n(X,x)\cong\pi_n(Y,y)$, using Corollary \ref{fibres_are_homotopic}, so the claim holds for $X$, too. 
\end{proof}
\begin{prop}\label{abelians_k1} Let $X$ be an abelian variety over $C$ and let $x\in X(C)$. Then $\pi_1(X,x)\cong\mathbb Z^{2\dim(X)}$ and $S_*(X,s_x)$ is weakly equivalent to $B\pi_1(X,x)$.
\end{prop}
\begin{proof} The proof of this claim is essentially the same as the proof of two claims above. Again we may assume without the loss of generality that $K$ contains $\mathbb R$.  Because of the existence of a connected smooth moduli stack for principally polarised abelian varieties of dimension $d=\dim(X)$, there is a smooth, projective map $f:T\to M$ of smooth varieties over $C$ such that $M$ is connected, the geometric fibres of $f$ are principally polarised abelian varieties of dimension $d$ and every such variety over $C$ is isomorphic to one of these fibres. Since over an algebraically closed field every abelian variety admits a principal polarisation, we get that in fact every abelian variety of dimension $d$ is the fibre of $f$ over some $C$-valued point of $M$. Let $Y$ be a abelian variety of dimension $d$ defined over the subfield
$\mathbb C\subseteq C$ and let $y\in Y(\mathbb C)$. Then we can deduce that
$$\pi_n(X,x)\cong\pi_n(Y,y)\cong\pi_n(S^{top}_*(Y),s_y)\ (\forall n\in\mathbb N)$$
by arguing as we did above, and the claim immediately follows.
\end{proof}

\section{The Riemann existence theorem}

\begin{defn}\label{6.1} Let $G$ be a group and let $\widehat G$ be its profinite completion. Following Serre (see page 13 of [11]) we will say that $G$ is good if the homomorphism of cohomology groups $H^n(\widehat G, M)\rightarrow H^n(G, M)$ induced by the natural homomorphism $G\rightarrow\widehat G$ is an isomorphism for every finite $G$-module $M$. Examples of good groups include finitely generated free groups, finitely generated free abelian groups, and the surface groups
$\Gamma_{g,0}$. (For an explanation of these well-known facts, see Examples 5.4 and 5.6 of \cite{Pa1}, for example.)
\end{defn}
\begin{lemma}\label{good_group} The group $H$ is good if there is a short exact sequence 
$$\xymatrix{1\ar[r] &N\ar[r] &H\ar[r] &G\ar[r] &1}$$
such that $G$, $N$ are good, $N$ is finitely generated, and the cohomology groups $H^q(N,M)$ are finite for every $q\in\mathbb N$ and every finite $H$-module $M$. In this case the sequence
$$\xymatrix{1\ar[r] &\widehat N\ar[r] &\widehat H\ar[r] &\widehat G\ar[r] &1}$$
is also exact. 
\end{lemma}
\begin{proof} The first claim is the content of Exercise 2(c) of [11] on page 16. The second claim is Exercise 2(b) of [11] on page 16. 
\end{proof}
\begin{defn} An elementary fibration over a field $F$ is a morphism of $F$-varieties $f:X\to S$ that can be embedded into a commutative diagram
$$\xymatrix{ X\ar[r]^{j}\ar[rd]_{f} & Y\ar[d]^{h} & Z\ar[l]_{i}\ar[ld]^{g}\\
 & S & }$$
in which:
\begin{enumerate}
\item[$(a)$] the map $j$ is an open immersion, the set $j(X)$ is dense in every fibre of $h$, and $Y=i(Z)
\cup j(X)$;
\item[$(b)$] the map $h$ is smooth and projective, with geometrically irreducible fibres of dimension $1$;
\item[$(c)$] the map $g$ is finite and \'etale, and each fibre of $g$ is nonempty.
\end{enumerate}
\end{defn}
\begin{defn} An elementary neighbourhood over a field $F$ is a smooth, quasi-projective $F$-variety $X$ whose structure map $X\to\textrm{\rm Spec}(F)$ may be factored as a composition of elementary fibrations:
\begin{equation}\label{6.4.1} 
X=X_n\to X_{n-1}\to\cdots\to X_1\to\textrm{\rm Spec}(F).
\end{equation}
In addition we say that $X$ is {\it affine} if all maps in (\ref{6.4.1}) are affine. 
\end{defn}
\begin{prop}\label{elementary1} Let $X$ be an affine elementary neighbourhood over $C$, and let $x\in X(C)$. Then $\pi_1(X,x)$ is good and $S_*(X,s_x)$ is weakly  equivalent to $B\pi_1(X,x)$.
\end{prop}
\begin{proof} We may assume that a sequence of maps as in (\ref{6.4.1}) is given. We are going to show that claim for $X=X_n$ by induction on $n$, that is, the dimension of $X$. Since $X_1$ is a smooth, affine, connected curve, so the claim follows from Propositions \ref{curves_fungroup} and \ref{curves_k1} and the remark at the end of Definition \ref{6.1} in the $n=1$ case. Now suppose that it holds for $X_{n-1}$. Let $f:X\to X_{n-1}$ be the first map in (\ref{6.4.1}), and fix an $x\in X(C)$. Since $f^{-1}(x)$ is a smooth, affine, connected curve, we get that $\pi_k(X_{n-1},f(x))=0$ for every $k\geq2$ by Proposition \ref{curves_k1}. Also $\pi_k(X_{n-1},f(x))=0$ for every $k\geq2$ by the induction hypothesis. So it follows from Corollary \ref{homotpy_exact_sequence} that
$\pi_k(X,f(x))=0$ for every $k\geq2$, and there is a short exact sequence:
$$\xymatrix{1\ar[r] &\pi_1(f^{-1}(x),x)\ar[r] &
\pi_1(X,x)\ar[r] &\pi_1(X_{n-1},f(x))\ar[r] &1.}$$
By the induction hypothesis the group $\pi_1(X_{n-1},f(x))$ is a good group. By Proposition \ref{curves_fungroup} the group
$\pi_1(f^{-1}(x),x)$ is a finitely generated free group, so it is also good. Since for every finitely generated free group $G$ and for every finite $G$-module $M$ the cohomology groups $H^k(G,M)$ are finite for all $k$, we get that $\pi_1(X,x)$ is good by Lemma \ref{good_group}. 
\end{proof}
\begin{notn} For every locally Noetherian scheme $X$ and geometric point $x$ of $X$ let $Et(X)$ denote the Artin-Mazur \'etale homotopy type of $X$, and let $Et(X,x)$ denote the pointed version of the Artin-Mazur \'etale homotopy type of the pair $(X,x)$. Moreover let $\pi_{n}^{et}(X,x)$ be the \'etale fundamental group of $X$ with base point $x$ for every positive integer $n$. For the sake of simple notation we will let $\widehat{\pi}_{1}(X,x)$ denote the profinite completion of
$\pi_{1}(X,x)$ for every pointed $lsa$-space $(X,x)$ over $K$.
\end{notn}
\begin{lemma}\label{good} Let $X$ be a smooth variety over $\mathbb C$, and let $x\in X(\mathbb C)$. Assume that
pointed topological space $(X(\mathbb C),x)$ has the homotopy type of the Eilenberg-MacLane space $B\pi_1(X(\mathbb C),x)$ and the group $\pi_1(X(\mathbb C),x)$ is good. Then $Et(X,x)$ is weakly homotopy equivalent to $B\pi^{et}_1(X,x)$.
\end{lemma}
\begin{proof} This is Proposition 5.2 of \cite{Pa2} on page 831.
\end{proof}
\begin{lemma}\label{notproper} Let $X$ be a smooth geometrically irreducible quasi-projective variety over an algebraically closed field $F$ of characteristic zero, let $x\in X(F)$, and let $L$ be another algebraically closed field containing $L$. Then $Et(X,x)$ is weakly homotopy equivalent to $Et(X_L,x)$.
\end{lemma}
\begin{proof} This is Proposition 5.4 of \cite{Pa2} on page 831.
\end{proof}
\begin{prop}\label{elementary2} Let $X$ be an affine elementary neighbourhood over an algebraically closed field $L$ of characteristic zero, and let $x\in X(L)$. Then the pointed \'etale homotopy type $Et(X,x)$ is weakly  equivalent to $B\pi^{et}_1(X,x)$.
\end{prop}
\begin{proof} There is a subfield $F\subset L$ which is finitely generated over $\mathbb Q$ and $X$ is already defined over $F$, that is, there is an affine elementary neighbourhood $Y$ over $F$ whose base change to $L$ is $X$, and $x$ is an
$F$-valued point in $Y$. Choose an embedding $i:\overline F\rightarrow\mathbb C$ of fields. Then the base change $Y_{\mathbb C}$ of the variety $Y_{\overline F}$ to $\mathbb C$ with respect to this embedding is also an affine elementary neighbourhood. By Proposition \ref{elementary1} and part
$(ii)$ of Theorem \ref{basic_comparison2} the pointed topological space $(Y_{\mathbb C}(\mathbb C),x)$ has the homotopy type of the Eilenberg-MacLane space $B\pi_1(Y_{\mathbb C}(\mathbb C),x)$, and its topological fundamental group is good. Therefore we get from Lemma \ref{good} that $Et(Y_{\mathbb C},x)$ is weakly homotopy equivalent to $B\pi^{et}_1(Y_{\mathbb C},x)$. By a repeated application of Lemma \ref{notproper} we get that $Et(X,x)=Et(Y_K,x)$ is weakly homotopy equivalent to $Et(Y_{\mathbb C},x)$ and hence $\pi^{et}_1(X,x)\cong\pi^{et}_1(Y_{\mathbb C},x)$ and so $Et(X,x)$ is weakly homotopy equivalent to $B\pi^{et}_1(X,x)$.
\end{proof}
We will need the following important result of Artin:
\begin{thm}[Artin]\label{artin_k1} Let $U$ be a smooth algebraic variety over an algebraically closed field $F$ and let $x\in U(F)$. Then there is a Zariski-open neighbourhood $X$ of $x$ which is an affine elementary neighbourhood. 
\end{thm}
\begin{proof} By Artin's classical result in SGA4 (see XI.3.3 of \cite{AGV}) there is a Zariski-open neighbourhood $X$ of $x$ which is an elementary neighbourhood. We only need to add that we may actually choose $X$ to be an affine elementary neighbourhood. We are going to prove the latter by induction on the dimension $n$ of $X$. The $n=1$ case is trivial. Now let $f:X\to X_{n-1}$ be an elementary fibration. By shrinking $X_{n-1}$ around $f(x)$, if it is necessary, we may assume that $f$ has a multisection whose image does not contain $x$. Replacing $X$ with the complement of the image of this multisection we reduce to the case when $f$ is affine. By the induction hypothesis there is a Zariski-open neighbourhood $X'_{n-1}$ of $f(x)$ in $X_{n-1}$ which is an affine elementary neighbourhood. By replacing
$f$ with the restriction of $f$ onto $f^{-1}(X'_{n-1})$ we may assume without the loss of generality that $X_{n-1}$ is an affine elementary neighbourhood. In this $X$ is also an affine elementary neighbourhood.
\end{proof}
\begin{notn} Let $X$ be a smooth variety over $C$. Let
$\mathcal C_X$ denote the category of finite \'etale covers of $X$, and let $\mathcal D_X$ denote the category of finite covers of $X_{lsa}$. Since for every finite \'etale cover $Y\to X$ the corresponding map $Y_{lsa}\to X_{lsa}$ of the associated $lsa$-spaces is locally trivial by Theorem \ref{thom_isotropy2}, we get a functor $\Psi_X:\mathcal C_X
\to\mathcal D_X$ which associates to every finite \'etale cover $Y\to X$ the finite cover $Y_{lsa}\to X_{lsa}$.
\end{notn}
The analogue of the classical Riemann existence theorem in our setting is the following result:
\begin{thm}\label{riemann_existence} Let $X$ be a smooth variety over $C$. The functor $\Psi_X:\mathcal C_X\to\mathcal D_X$ is an equivalence of categories.
\end{thm}
This theorem is proved for all schemes of finite type over $C$ in Huber's thesis (see Satz 12.12 on page 184 of \cite{Hu}). Our proof, which only works in the smooth case, is very different, and does not use coherent sheaves, nor GAGA-type theorems. Before we start proving this theorem we will reformulate it in an equivalent form. Note that we only need to show the claim in the case when $X$ is connected. Assume that this is the case, let $x\in X(C)$, and let $F:\mathcal D_X\to Sets$ be the functor which associates to every finite cover
$Y$ of $X_{lsa}$ the set of points of $Y$ over $x$. This functor is pro-representable by a pro-object $\mathbf U(X,x)$ which is unique up to a unique isomorphism. Since every finite cover of $X_{lsa}$ is a quotient of the universal cover by a subgroup of finite index, the profinite automorphism group of $\mathbf U(X,x)$ is just profinite completion $\widehat{\pi}_{1}(X,x)$ of $\pi_{1}(X,x)$. Therefore we have a natural homomorphism:
$$\psi_X:\widehat\pi_{1}(X,x)\to\pi_{1}^{et}(X,x)$$
of profinite groups. By the above Theorem \ref{riemann_existence} is equivalent to the following
\begin{thm}\label{riemann_existence2} Let $X$ be a smooth connected variety over $C$, and let $x\in X(C)$. Then
$\psi_X:\widehat\pi_{1}(X,x)\to\pi_{1}^{et}(X,x)$ is an isomorphism.\qed
\end{thm}
Now we can start the
\begin{proof}[Proof of Theorem \ref{riemann_existence}] Let $Y$ and $Z$ be two finite \'etale covers of $X$.
\begin{lemma}\label{faithful} The canonical map:
\begin{equation}\label{3.2.1b}
\mathrm{Hom} _{\mathcal C_X}(Y,Z) \to
\mathrm{Hom} _{\mathcal D_X}(Y_{lsa},Z_{lsa})
\end{equation}
furnished by $\Psi_X$ is bijective.
\end{lemma}
\begin{proof} We may assume without the loss of generality that $Y$ is connected. Giving an $X$-morphism of $Y$ into $Z$ is equivalent to giving a connected component $W$ of $Y\times_{X}Z$ such that the morphism $W\to Y$ induced by the first projection is an isomorphism. As the connected components of $Y\times_{X}Z$ correspond bijectively to the connected components of $Y_{lsa}\times_{X_{lsa}}Z_{lsa}$ by Proposition \ref{connectivity}, and the morphism $W\to Y$ is an isomorphism if and only if the same holds for $W_{lsa}\to Y_{lsa}$, the bijectivity of the map in (\ref{3.2.1b}) follows.
\end{proof}
Next we are going to prove that the functor $\Psi_X$ is  essentially surjective. First note that the claim is local for the Zariski topology. Indeed let $f:\mathcal Y\to X_{lsa}$ be a finite cover; we need to prove that there is a finite \'etale cover $Y$ of $X$ such that there is an isomorphism $Y_{lsa}\cong\mathcal Y$ in $\mathcal D_X$. Assume that there is a cover $\{X(i)\mid i\in I\}$ of $X$ by open subschemes such that for every $i\in I$ there is a finite \'etale cover $Y(i)$ of $X$ such that there is an isomorphism $Y(i)_{lsa}\cong f^{-1}(X(i)_{lsa})$. By Lemma \ref{faithful} these $Y(i)$ patch together to a descent data over $X$, which is effective, since the maps $Y(i)\to X$ are affine. The corresponding scheme is the $Y$ sought for. Now we are going to show that $\Psi_X$ is essentially surjective in various special cases. 
\begin{lemma}\label{classical_riemann} Let $X$ be a connected variety defined over $\mathbb Q^{real}(\mathbf i)$. Then the map $\psi_{X_C}:\widehat\pi_{1}(X_C,x)\to\pi_{1}^{et}(X_C,x)$ is an isomorphism for every $x\in X(C)$.
\end{lemma}
\begin{proof} As we already noted there is a real closed extension $L$ of $K$ which contains a copy of $\mathbb R$ by Lemma \ref{completion}. There is a commutative diagram:
\begin{equation}\label{6.15.1}
\xymatrix{ \widehat\pi_{1}(X_C,x)\ar[r]^{\psi_{X_C}}\ar[d] &
\pi_{1}^{et}(X_C,x) \ar[d] \\
\widehat\pi_{1}(X_{L(\mathbf i)},x)
\ar[r]^{\psi_{X_{L(\mathbf i)}}} &
\pi_{1}^{et}(X_{L(\mathbf i)},x)}
\end{equation}
furnished by base change. The first vertical map is an isomorphism by Theorem \ref{basic_comparison2}. The second map is also an isomorphism which can be seen as follows. By Hironaka's resolution of singularities there is a projective variety $Y$ over $C$ which contains $X_C$ as an open subvariety such that the complement $D\subset Y$ is a normal crossings divisor. By the tame invariance theorem the tame \'etale fundamental groups $\pi_1^D(Y,x)$ and $\pi_1^{D_{L(\mathbf i)}}(Y_{L(\mathbf i)},x)$ are isomorphic. But since the base fields have characteristic zero we have $\pi_1^D(Y,x)\cong\pi^{et}_1(X_C,x)$ and $\pi_1^{D_{L(\mathbf i)}}(Y_{L(\mathbf i)},x)\cong\pi^{et}_1(X_{L(\mathbf i)},x)$. Therefore $\pi^{et}_1(X_C,x)\cong\pi^{et}_1(X_{L(\mathbf i)},x)$.

Therefore it will be enough to show that $\widehat\pi_{1}(X_{L(\mathbf i)},x)\to\pi_{1}^{et}(X_{L(\mathbf i)},x)$ is an isomorphism for every $x\in X(C)$. In other words we may assume that $K$ contains $\mathbb R$. Using the same type of diagram as (\ref{6.15.1}) for $K$ and $\mathbb R$ we get that it will be enough to show the claim $K=\mathbb R$. As we explained above the validity of the claim is independent of the choice of $x$, and hence we may assume that $x\in X(\mathbb Q^{real}(\mathbf i))$. In this case we have a diagram:
$$\xymatrix{ \widehat\pi_{1}(S^{top}_*(X_{\mathbb C}),s_x)\ar[r]^{\ \ \ \tau_{X_{\mathbb C}}}
\ar@/^1.5pc/[rr]^{\psi_{X_{\mathbb C}}\circ\tau_{X_{\mathbb C}}} & 
\widehat\pi_{1}(X_{\mathbb C},x)\ar[r]^{\!\!\psi_{X_{\mathbb C}}} &
\pi_{1}^{et}(X_{\mathbb C},x),}$$
where $\widehat\pi_{1}(S^{top}_*(Y),s_y)$ is the profinite completion of $\pi_{1}(S^{top}_*(Y),s_y)$. 
The map $\tau_{X_{\mathbb C}}$ is an isomorphism by part $(ii)$ of Theorem \ref{basic_comparison2}, while the composition $\psi_{X_{\mathbb C}}\circ\tau_{X_{\mathbb C}}$ is an isomorphism by the classical Riemann existence theorem in SGA4 (see XI.4.3 of \cite{AGV}). Therefore
$\psi_{X_{\mathbb C}}$ must be an isomorphism, too.
\end{proof}
\begin{defn} For every positive integer let $A[n]$ denote following open variety of the affine space $\mathbb A^n$ over $\mathbb Q$:
$$A[n]=\{(x_1,x_2,\ldots,x_n)\in \mathbb A^n\mid
x_i\neq x_j\ (\forall i\neq j,\ i,j\in\{1,2,\ldots,n\})\}.$$
We have an affine map
$$\pi_n:A[n]\to A[n-1],\quad(x_1,x_2,\ldots,x_n)\to
(x_1,x_2,\ldots,x_{n-1})$$
for each $n\geq2$. This map is an elementary fibration, and hence we get that $A[n]$ is an affine elementary neighbourhood by induction. Moreover the comparison map $\widehat\pi_{1}(A[n]_C,x)\to\pi_{1}^{et}(A[n]_C,x)$ is an isomorphism for every $x\in A[n](C)$ by Lemma \ref{classical_riemann}. 
\end{defn}
\begin{lemma}\label{curve_riemann1} Let $X$ be a connected affine rational curve over $C$, and let $x\in X(C)$. Then the comparison map $\widehat\pi_{1}(X,x)\to\pi_{1}^{et}(X_C,x)$ is an isomorphism.
\end{lemma}
\begin{proof} By assumption $X=\mathbb A^1_C-S$, where
$S\subset\mathbb A^1(C)$ is a set of cardinality $n$ for some natural number $n$. Since $\pi^{et}_1(\mathbb A^1_C,x)$ is trivial for any $x\in\mathbb A^1(C)$, the claim follows from Proposition \ref{curves_fungroup} in the case $n=0$. Now assume that $n$ is positive. Then there is a $y\in A[n](C)$ such that $X$ is isomorphic to the fibre of $\pi_{n+1}:A[n+1]\to A[n]$ over $y$. We may identify $X$ with this fibre. As we saw in the proof of Proposition \ref{elementary1} there is a short exact sequence:
$$\xymatrix{1\ar[r] &\pi_1(X,x)\ar[r] &
\pi_1(A[n+1],x)\ar[r] &\pi_1(A[n],y)\ar[r] &1.}$$
We have also shown there that the conditions of Lemma \ref{good_group} are satisfied, and therefore the short sequence
$$\xymatrix{1\ar[r] &\widehat\pi_1(X,x)\ar[r] &
\widehat\pi_1(A[n+1],x)\ar[r] &\widehat\pi_1(A[n],y)\ar[r] &1}$$
we got via profinite completion is also exact by the second part of Lemma \ref{good_group}. On the other hand by Proposition \ref{elementary2} we have $\pi^{et}_2(A[n],y)=0$, so the homotopical exact sequence for the \'etale homotopy type (see Theorem 11.5 of \cite{Fr} on page 107) applied to the map
$\pi_{n+1}$ furnishes a short exact sequence:
$$\xymatrix{1\ar[r] &\pi^{et}_1(X,x)\ar[r] &
\pi^{et}_1(A[n+1],x)\ar[r] &\pi^{et}_1(A[n],y)\ar[r] &1.}$$
Consequently we have a commutative diagram of short exact sequences:
$$\xymatrix{ 1\ar[r] &
\widehat\pi_{1}(X,x)\ar[r]\ar[d]^{\psi_X} &
\widehat\pi_{1}(A[n+1],x)\ar[r]\ar[d]^{\psi_{A[n+1]}} &
\widehat\pi_{1}(A[n],y) \ar[r]\ar[d]^{\psi_{A[n]}} &1 \\ 
1\ar[r] &
\pi_{1}^{et}(X,x)\ar[r] &
\pi_{1}^{et}(A[n+1],x)\ar[r] & \pi_{1}^{et}(A[n],y)
\ar[r] &1.}$$
By Lemma \ref{classical_riemann} the middle and last vertical maps are isomorphisms, so the first map must be an isomorphism, too.
\end{proof}
\begin{lemma}\label{curve_riemann2} Let $X$ be a connected quasi-projective algebraic curve over $C$, and let $x\in X(C)$. Then $\psi_C:\widehat\pi_{1}(X,x)\to\pi_{1}^{et}(X_C,x)$ is an isomorphism.
\end{lemma}
\begin{proof} In this case it will be simpler to show that the functor $\Psi_X$ is essentially surjective. First assume that there is a finite \'etale cover $f:X\to Y$ where $Y$ is a connected affine rational curve over $C$. Let $\gamma:\mathcal Z\to X_{lsa}$ be a finite cover. The composition of $\gamma$ and the map $f_{lsa}:X_{lsa}\to Y_{lsa}$ of $lsa$-spaces underlying $f$ is a finite cover of $Y_{lsa}$, so by Lemma \ref{curve_riemann1} there is a finite \'etale cover $h:Z\to Y$ such that $Z_{lsa}$ and $\mathcal Z$ are isomorphic in $\mathcal D_Y$. By Lemma \ref{faithful} the map
$$\mathrm{Hom} _{\mathcal C_Y}(Z,X) \to
\mathrm{Hom} _{\mathcal D_Z}(Z_{lsa},X_{lsa})$$
furnished by $\Psi_Y$ is bijective. Therefore there is a morphism $g:Z\to X$ such that the underlying map of $lsa$-spaces is
$\gamma$. Since both $h$ and $f$ are \'etale, and
$h=f\circ g$, the same holds for $g$.

Now let us consider the general case. Since the problem is local on $X$ by the above, it will be sufficient to show that $X$ can be covered by open subschemes which are finite \'etale covers of connected affine rational curves over $C$. We are going to use an argument similar to the proof of Lemma \ref{conn2}. Let $g$ be the genus of the unique smooth, projective, irreducible curve $\widetilde X$ containing $X$ as an open subscheme. Now let $x\in X(C)$ be arbitrary. Let $D$ be a divisor on $\widetilde X$ whose contains the complement of $X$ in $\widetilde X$, but does not contain $x$. By a routine application of the Riemann--Roch theorem we get that the full linear system of $\mathcal O(D)$ furnishes a projective embedding of $\widetilde X$ into $\mathbb P^{g+3}_C$. Now let $x\in X(C)$ be arbitrary. By Bertini's hyperplane section theorem there is a hyperplane $H\subset\mathbb P^{g+3}_C$ which intersects $\widetilde X$ transversally, contains $x$, and does not contain any point of the complement $\widetilde X-X$. The pencil spanned by $D$ and $H\cap\widetilde X$ furnishes a map $f:\widetilde X\to\mathbb P^1_C$ such that the pre-image of $0$ is $D$, the pre-image of $\infty$ is $H\cap\widetilde X$, and $f$ is \'etale at each point of $H\cap\widetilde X$. Therefore there is a non-empty open subcurve $U\subseteq\mathbb P^1_C$ which contains $f(x)$, and the restriction of $f$ onto $f^{-1}(U)$ is a finite, \'etale map onto $U$. 
\end{proof}
Now we are going to prove that the functor $\Psi_X$ is  essentially surjective in general. By the above the question is local over $X$, and so we may suppose that $X$ is an affine elementary neighbourhood by Theorem \ref{artin_k1}. We may assume that a sequence of maps as in (\ref{6.4.1}) is given. We are going to show the claim for $X=X_n$ by induction on $n$, that is, on the dimension of $X$. Since $X_1$ is a smooth, affine, connected curve, the claim follows from Lemma \ref{curve_riemann2}. Now suppose that it holds for $X_{n-1}$. Let $f:X\to X_{n-1}$ be the first map in (\ref{6.4.1}), and set $y=f(x)$. As we saw in the proof of Proposition \ref{elementary1} there is a short exact sequence:
$$\xymatrix{1\ar[r] &\pi_1(f^{-1}(y),x)\ar[r] &
\pi_1(X,x)\ar[r] &\pi_1(X_{n-1},y)\ar[r] &1.}$$
We have also shown there that the conditions of Lemma \ref{good_group} are satisfied, and therefore the short sequence
$$\xymatrix{1\ar[r] &\widehat\pi_1(f^{-1}(y),x)\ar[r] &
\widehat\pi_1(X,x)\ar[r] &\widehat\pi_1(X_{n-1},y)\ar[r] &1}$$
we got via profinite completion is also exact by the second part of Lemma \ref{good_group}. On the other hand by Proposition \ref{elementary2}  we have $\pi_2(X_{n-1},y)=0$, so the homotopical exact sequence for the \'etale homotopy type (see Theorem 11.5 of \cite{Fr} on page 107) applied to $f$ furnishes a short exact sequence:
$$\xymatrix{1\ar[r] &\pi_{1}^{et}(f^{-1}(y),x)\ar[r] &
\pi_{1}^{et}(X,x)\ar[r] &\pi_{1}^{et}(X_{n-1},y)\ar[r] &1.}$$
Consequently we have a commutative diagram of short exact sequences:
$$\xymatrix{ 1\ar[r] &
\widehat\pi_{1}(f^{-1}(y),x)\ar[r]\ar[d]^{\psi_{f^{-1}(y)}} &
\widehat\pi_{1}(X,x)\ar[r]\ar[d]^{\psi_X} &
\widehat\pi_{1}(X_{n-1},y) \ar[r]\ar[d]^{\psi_{X_{n-1}}} &1 \\ 
1\ar[r] &
\pi_{1}^{et}(f^{-1}(y),x)\ar[r] &
\pi_{1}^{et}(X,x)\ar[r] & \pi_{1}^{et}(X_{n-1},y)
\ar[r] &1.}$$
By the induction hypothesis the first and last vertical maps are isomorphisms, so the middle map must be an isomorphism, too.
\end{proof}

\section{Cohomological and homotopical comparison theorems}

Let $X$ be a nonsingular variety. Then $X$ can be endowed in a natural way with the structure of a complex manifold. I write $X_{an}$ for $X$ regarded as a complex manifold and $X_{cx}$ for $X$ regarded as a topological space with the complex topology (thus $X_{an}$ is $X_{cx}$ together with a sheaf of rings). By Riemann existence theorem in the previous section there is a natural one-to-one correspondence $F\mapsto F_{cx}$ between the locally constant sheaves $F$ on $X_{et}$ with finite stalks and the locally constant sheaves on $X_{cx}$ with finite stalks. The main theorem of this section is
\begin{thm}\label{cohomological} Let $X$ be a connected nonsingular variety over $\mathbb C$. For any locally constant sheaf $F$ on $X_{et}$ with finite stalks we have $H^r(X_{et},F)\cong H^r(X_{cx},F_{cx})$ for every $r\geq0$.
\end{thm}
This theorem is proved for all schemes of finite type over $C$ in Huber's thesis (see for example Satz 12.14 on page 187 and part $\gamma$ of the Satz on page 192 in \cite{Hu}, where such claims of increasing generality are stated). In particular Satz 12.14 on page 187 of \cite{Hu} stated very similarly, but only for constant sheaves. However it is easy to reduce the theorem above to this case by taking a trivialising finite \'etale cover and compare the resulting Hochschild--Serre spectral sequences both over $X_{et}$ and $X_{cx}$. Our proof, which only works in the smooth case, is very different, and does not use coherent sheaves, nor GAGA-type theorems.
\begin{proof} Let $X_{ecx}$ denote $X$ endowed with the Grothendieck topology for which the coverings are surjective families of \'etale maps
$U_i\to U$ of complex manifolds over $X$. There are continuous morphisms (of sites):
$$\xymatrix{X_{cx} & X_{ecx}\ar[l]\ar[r] & X_{et}.}$$
The inverse mapping theorem shows that every complex-\'etale covering $(U_i\to U)$ of a complex manifold $U$ has a refinement that is an open covering (in the usual sense). Therefore, the left hand arrow gives isomorphisms on cohomology. It remains to prove that the right hand arrow does also. For this we can use an entirely standard argument as follows. Consider the following situation: $X$ is a set with two topologies, $T_1$ and $T_2$ (in the conventional sense), and assume that $T_2$ is finer than $T_1$. Let $X_i$ denote $X$ endowed with the topology $T_i$. Because $T_2$ is finer than $T_1$, the identity map $f:X_2\to X_1$ is continuous. We ask the question: for a sheaf $F$ on $X_2$, when is $H^r(X_1,f_*F)\cong H^r(X_2,F)$ for every $r\geq0$. Note that $f_*F$ is simply the restriction of $F$ to $X_1$. An answer is given by the Leray spectral sequence
$$H^r(X_1,R^sf_*F)\Rightarrow H^{r+s}(X_2,F).$$
Namely, the cohomology groups agree if $R^sf_*F=0$ for every $s>0$. But $R^sf_F$ is the sheaf associated with the presheaf $U\mapsto H^s(U_2,F)$ (here $U$ is an open subset of $X_1$, and $U_2$ denotes the same set endowed with the $T_2$-topology). Thus we obtain the following criterion for $H^r(X_1,f_*F)\cong H^r(X_2,F)$: for any open $U\subseteq X_1$ and any $t\in H^s(U_2,F)$, where $s>0$, there exists a covering $U=\bigcup U(i)$ (for the $T_1$-topology) such that $t$ maps to zero in $H^s(U(i)_2,F)$ for every $i$. 

 and for this, the above discussion shows that it suffices to prove the following statement:
\end{proof}
\begin{lemma}\label{4.2} Let $U$ be a connected nonsingular variety, and let $F$ be a locally constant sheaf on $U_{ecx}$ with finite stalks. For every $t\in H^s(U_{cx},F)$, where $s>0$, there exists an \'etale covering $U(i)\to U$ (in the algebraic sense), such that $t$ maps to zero in $H^s(U(i)_{cx},F)$ for each $i$.
\end{lemma}
\begin{proof} We use induction on the dimension of $U$. Clearly the problem is local on $U$ for the \'etale topology, i.e., it suffices to prove the statement for the image $t_i$ of $t$ in $H^r(U(i),F)$ for each $U(i)$ in some \'etale covering $(U(i)\to U)_{i\in I}$ of $U$. Thus we can assume that $F$ is constant, and that $U$ has been replaced by some Zariski-small set $U$. We now prove Lemma \ref{4.2}. Because the statement is local for the \'etale topology on $U$, we may assume that $U$ admits an elementary fibration and that $F$ is constant: $F=\Lambda$. In this case we can use that $U$ is an Eilenberg-McLane space, so its cohomology can be killed by a cover, which is algebraic by the Riemann existence theorem. The preceeding discussion shows that this completes the proof of Theorem \ref{cohomological}.
\end{proof}

\section{The Artin-Mazur comparison theorem}

\begin{defn} Let $X$ be a connected, pointed lsa space. Assume that every admissible open subset of $X$ is paracompact, and that $X$ is locally contractable, i.e., that every point $x\in X$ contains arbitrarily small contractable neighbourhoods. Such a space is locally
arc-wise connected. Let $\mathcal C$ be the ordinary site on $X$ of coproducts of open subsets. We denote by $S_*X$ the simplical set of singular simplexes of $X$.
\end{defn}
\begin{thm}\label{5.2} Let $U.$ be an admissible hypercovering of $\mathcal C$ such that every connected component of every $U_n$ is contractible. Then the simplicial set $\pi_0(U.)$ is isomorphic to the simplicial set $S_*X$ in $\textrm{\rm Ho}(SSets)$.
\end{thm}
Because those hypercoverings $U.$ of $\mathcal C$ such that every connected component of every $U_n$ is contractible are cofinal by assumption, we have the following immediate
\begin{cor} The pro-object $\Pi(\mathcal C)$ is canonically isomorphic to the element $S_*X$ in $\textrm{\rm Pro}-\textrm{\rm Ho}(SSets)$.\qed
\end{cor}
\begin{proof}[Proof of Theorem \ref{5.2}] Let $S_*U_n$ denote the singular complex of $U_n$. Then $S_*U.$ is a bisimplicial object in $Sets$. We denote by $(DU).$ its diagonal simplicial set $(DU)_n=S_nU_n$. Then we have obvious maps of simplicial sets:
$$\xymatrixcolsep{5pc}\xymatrix{
&(DU).\ar[ld]_{\alpha} \ar[rd]^{\beta} & \\
\pi_0(U.) & & S_*X,}$$
and we claim that these two maps are homotopy equivalences in $\textrm{\rm Ho}(SSets)$, which will prove the theorem. First we are going to show that the category of simplicial covering spaces of the three objects are equivalent via $\alpha,\beta$. Since $X$ is locally arc-wise connected, we can identify simplicial covering spaces of $S_*X$ with locally trivial covering spaces of $X$. Now by the discussion of section 10 of Artin--Mazur, simplicial covering spaces of $\pi_0(U.)$ are given by gluing data for a covering space of $X$ relative to the open covering $U_0$. Thus it is trivial that coverings of $\pi_0(U.)$ and of $S_*X$ correspond. It remains to show for instance that every descent data for a simplicial covering of $A.$ is obtained from descent data on $\pi_0(U.)$ in a unique way. This is a simple exercise which we leave to the reader.

Now let $Y\to X$ be a covering space, say connected. Then $U'.=Y\times_XU.$ is obviously a hypercovering of $Y$ satisfying the $J$ assumptions of the theorem, and if we form the simplicial covering space $\pi_0(U'.)\times_{\pi_0(U.)}(DU).$, it is in fact the diagonal simplicial set $DU'_q=S_qU'_q$ of $S_*U.$. Moreover the covering space of $S_*X$ corresponding to $Y$ is just $S_*Y$. Therefore we obtain a diagram
$$\xymatrixcolsep{5pc}\xymatrix{
&(DU').\ar[ld]_{\alpha} \ar[rd]^{\beta} & \\
\pi_0(U'.) & & S_*Y,}$$
covering the diagram above. To show that $\alpha,\beta$ are homotopy equivalences, it suffices to show that for every such $Y$, the maps $\alpha,\beta$ induce isomorphisms on cohomology with arbitrary coefficients $A$. Thus we may as well drop the $'$.

By the Eilenberg-Zilber theorem, the two spectral sequences
$$'E^{pq}_2=H^p_v(H^q_h(A^{S_*U.}))\Rightarrow H^{p+q}((DU).,A)$$
$$''E^{pq}_2=H^p_v(H^q_h(A^{S_*U.}))\Rightarrow H^{p+q}((DU).,A)$$
where $h$ means horizontal and $v$ means vertical. Since $U_q$ is component-wise contractable, we have $H^q(U_p,A)={}' E^{pq}_1=0$, if $q>0$. Hence the spectral sequence yields
$$'E^{p0}_2=H^p(H^0(U_p,A))=H^p(\pi_0(U.),A)\cong H'((DU).,A).$$
This shows that $\alpha$ is a homotopy equivalence. To show that $\beta$ induces an isomorphism on cohomology, we use the sheaf of singular cochains $\mathcal S^q(.,A)$. The cosimplicial sheaf $S^*(A)$ is a flabby resolution of the constant sheaf $\underline A$, and it is used to give the usual identification of sheaf cohomology with singular cohomology.

Consider the surjective map of bicomplexes
$$A^{S_pU_q}\longrightarrow\mathcal S^q(.,A)\longrightarrow 0.$$
A consideration of the first spectral sequences $'E^{pq}_2$ above shows that $\phi$ induces an isomorphism on the cohomology of the associated total complexes. Consider the second spectral sequences: We have $H^q(\mathcal S^p(U.,A))\cong H^q(X,\mathcal S^p(A))$. This follows from the spectral sequence (8.15) of \cite{AM} since $\mathcal S^p(A)$ is flabby, and moreover $H^q(X,\mathcal S^p(A))=0$ if $q>0$. Thus the map
$\phi$ on $''E^{pq}_2$ yields
\begin{equation}\label{12.4}
''E^{p0}_2= H^p_h(H^0_v(A^{S_*U.}))\longrightarrow H^p(X,A),
\end{equation}
and the groups $H^p(X,A)$ are the abutments of the two spectral sequences. Now we are interested in the map
$$A^{S_*X}\longrightarrow A^{S_*U.}$$
given by the projection $S_*U_q\to S_*X$. This map induces a map
$$H^p(X,A)\longrightarrow ''E^{p0}_2.$$
Clearly the composition of this map with (\ref{12.4}) is the identity, which shows that $\beta$ induces an isomorphism on cohomology, and completes the proof.
\end{proof}
\begin{thm} Let $X$ be a connected, pointed scheme of finite type over $C$. Then there is a canonical map $e:X_{cl}\to X_{et}$, and it induces an isomorphism on profinite completions.
\end{thm}
\begin{proof} Let $\mathcal C'$ denote the small \'etale site of $X$ and let $\mathcal C$ be the ordinary (Grothendieck) site on the coproducts of admissible open subsets of $X(C)$ with respect to its usual topology. In addition to these sites we also introduce the site $\mathcal C''$ whose objects are lsa spaces $Y$ lying over the lsa space of $X_{cx}$ such that the map $Y \rightarrow X$ is a local isomorphism, i.e., that every point $x\in Y$ has a neighborhood which is isomorphic onto its image. Since any \'etale map of schemes $Y \rightarrow X$ over $\overline K$ is a local isomorphism on the underlying topological spaces, and since an open set is in $\mathcal C''$, we have morphisms of sites:
$$\xymatrixcolsep{5pc}\xymatrix{
&\mathcal C''\ar[ld]_{m} \ar[rd]^{m'} & \\
\mathcal C & & \mathcal C'.}$$
Now it is clear from the definition of local isomorphisms that every hypercovering of $\mathcal C''$ is dominated by a hypercovering of $\mathcal C$. Thus the map $\Pi(m):
\Pi(\mathcal C'')\rightarrow\Pi(\mathcal C)$ is a homotopy equivalence in $\textrm{\rm Pro}-\textrm{\rm Ho}(SSets)$, and so
$$\Pi(m)(E\Gamma):\Pi(\mathcal C'')(E\Gamma)\rightarrow
\Pi(\mathcal C)(E\Gamma)$$
is a bijection, and so $m'$ yields the map $e$. By the discussion of section 10 in \cite{AM} and the comparison theorems of Huber, this map $e$ induces an isomorphism on cohomology with twisted finite coefficients, and on profinite completion of the \'etale fundamental group. Therefore $e$ is a $\natural$-isomorphism, by (4.3) of \cite{AM}. Now apply Theorem (12.5) of \cite{AM} to obtain an actual isomorphism.
\end{proof}
\begin{cor} Suppose in addition that $X$ is geometrically unibranch. Then $X_{et}$ is isomorphic to the profinite completion of $X_{cl}$.
\end{cor}
\begin{proof} Apply (11.1) of \cite{AM}.
\end{proof}


\begin{thebibliography}{99}
\bibitem{AGV} M.~Artin, A.~Grothendieck, and
J.-L.~Verdier, {\it Th\'eorie des topos et cohomologie \'etale des sch\'emas (SGA4)}, Lecture Notes in Mathematics \textbf{269}, \textbf{270}, \textbf{305}, Springer, Berlin-New York, 1972.

\bibitem{AM} M.~Artin and B.~Mazur, {\it \'Etale homotopy}, Lecture Notes in Mathematics \textbf{100}, Springer, Berlin-New York, 1969.

\bibitem{BLR} S.~Bosch, W.~L{\"u}tkebohmert, and
M.~Raynaud, {\it N\'eron models}, Ergebnisse der Mathematik und ihrer Grenzgebiete \textbf{21}, Springer, Berlin-New York, 1990.

\bibitem{BCR} J.~Bochnak, M.~Coste and M.-F.~Roy, {\it Real algebraic geometry}, Ergebnisse der Mathematik und ihrer Grenzgebiete \textbf{36}, Springer, Berlin-New York, 1998. 

\bibitem{Bour} N.~Bourbaki, {\it Alg\`ebre}, 2nd ed., Actualit\'es Sci. Indust. \textbf{1179}, Hermann, Paris, 1964.

\bibitem{CS} M.~Coste and M.~Shiota, {\it Nash triviality in families of Nash manifolds}, Invent. Math. \textbf{108} (1992), 349--368.

\bibitem{DeKn1} H.~Delfs and M.~Knebusch, {\it An introduction to locally semialgebraic spaces}, Ordered fields and real algebraic geometry (Boulder, Colo., 1983), Rocky Mountain J. Math. \textbf{14} (1984), 945--963.

\bibitem{DeKn2} H.~Delfs and M.~Knebusch, {\it Locally semialgebraic spaces}, Lecture Notes in Mathematics \textbf{1173}, Springer, Berlin-New York, 1985. 

\bibitem{EW} M.~Edmundo and A.~Woerheide, {\it Comparison theorems for o-minimal singular (co)homology}, Trans. Amer. Math. Soc. \textbf{360} (2008), 4889--4912.

\bibitem{EP} M.~Edmundo and L.~Prelli, {\it The six Grothendieck operations on $o$-minimal sheaves}, arXiv:1401.0846.

\bibitem{Es} J.~Escribano, {\it Nash triviality in families of Nash mappings}, Ann. Inst. Fourier \textbf{51} (2001), 1209--1228.

\bibitem{Fr} E. M. Friedlander, {\it \'Etale homotopy of simplicial schemes}, Annals of Mathematical Studies, vol. 104, Princeton University Press, Princeton, 1982.

\bibitem{GJ} P.~G.~Goerss and J.~F.~Jardine, {\it Simplicial homotopy theory}, Modern Birkh\"auser Classics, Reprint of the 1999 edition, Birkh\"auser Verlag AG, Basel-Boston-Berlin, 2009.

\bibitem{Hu} R.~Huber, {\it Isoalgebraische R\"aume}, Dissertation, University of Regensburg, 1984, 231 pp. 

\bibitem{Hu2} R.~Huber, {\it Ein semialgebraischer Beweis der topologischen Form des Hauptsatzes von Zariski}, Manuscripta Mathematica
\textbf{61} (1988), 49--62.

\bibitem{HK} R.~Huber and M.~Knebusch, {\it A glimpse at isoalgebraic spaces}, Note Mat. \textbf{10} (1990), 315--336. 

\bibitem{Kn} M.~Knebusch, {\it Weakly semialgebraic spaces}, Lecture Notes in Mathematics \textbf{1367}, Springer, Berlin-New York, 1989. 

\bibitem{Pa1} A.~P\'al, {\it The real section conjecture and Smith's fixed point theorem for pro-spaces}, J. London Math. Soc. \textbf{3} (2011), 353--367.

\bibitem{Pa2} A.~P\'al, {\it \'Etale homotopy equivalence of rational points on algebraic varieties}, Algebra Number Theory \textbf{9} (2015), 815--873.

\bibitem{Sch} C.~Scheiderer, {\it Real and \'etale cohomology}, Lecture Notes in Mathematics \textbf{1588}, Springer, Berlin-New York, 1994.

\bibitem{Wo} A.~Woerheide, {\it $O$-minimal homology}, Thesis, University of Illinois at Urbana-Champaign, 1996, 56 pp.
\end{thebibliography}
\end{document}